\documentclass[10pt]{amsart}

\usepackage{amsmath}
\usepackage{amsfonts}
\usepackage{amssymb}
\usepackage{amsthm}
\usepackage{indentfirst}
\usepackage[OT2,T1]{fontenc}
\usepackage{times}

\usepackage{enumerate}

\newtheorem{theorem}{Theorem}[section]
\newtheorem{lemma}[theorem]{Lemma}
\newtheorem{proposition}[theorem]{Proposition}

\newtheorem{definition}[theorem]{Definition}

\newtheorem{assumption}[theorem]{Assumption}

\newtheorem{remark}[theorem]{Remark}
\newtheorem{notation}[theorem]{Notation}
\newtheorem{example}[theorem]{Example}
\newcommand{\bfF}{{\bf F}}
\newcommand{\bfV}{{\bf V}}
\newcommand{\Gal}{\operatorname{Gal}}
\newcommand{\ds}{\displaystyle}
\newcommand{\Tr}{\operatorname{Tr}}

\newcommand{\Sel}{\operatorname{Sel}}
\newcommand\Selr{\Sel_{\text{rel}}}
\newcommand{\Hom}{\operatorname{Hom}}

\newcommand\rat{{\mathbb Q}}

\newcommand\Col{\operatorname{Col}}

\newcommand\rank{\operatorname{rank}}

\newcommand\corank{\operatorname{corank}}

\newcommand\Proj{\operatorname{Proj}}
\newcommand\Qp{{\mathbb Q_p}}
\newcommand\Zp{{\mathbb Z_p}}

\newcommand\Z{\mathbb Z}

\newcommand{\bfT}{\mathbf T}
\newcommand{\bfA}{\mathbf A}

\newcommand{\Dieudonne}{\mathbf D}
\newcommand{\Dieu}{\Dieudonne}

\newcommand\mfp{\mathfrak p}

\newcommand\mm{\mathfrak m}

\newcommand\bfP{\mathbf P}

\newcommand\bfx{\mathbf x}
\newcommand\bfy{\mathbf y}
\newcommand\bfm{\mathbf m}
\newcommand\bfl{\mathbf l}

\newcommand\Q{\mathbb Q}

\newcommand\C{\mathbb C}

\newcommand\bfL{\mathbf L}
\newcommand\bfLalg{\bfL_{alg}}

\newcommand\OO{\mathcal O}

\newenvironment{nouppercase}{%
  \renewcommand{\uppercasenonmath}[1]{}}{}

\DeclareSymbolFont{cyrletters}{OT2}{wncyr}{m}{n}
\DeclareMathSymbol{\Sha}{\mathalpha}{cyrletters}{"58}

\newcommand\Label{\label}

\title[Rational Points of Abelian Varieties with Non-Ordinary Reduction]{Ranks of the Rational Points of Abelian Varieties over Ramified Fields, and Iwasawa Theory for Primes with Non-Ordinary Reduction}
\author[B. D. Kim]{Byoung Du (B. D.) Kim}
\address{Victoria University of Wellington, Wellington, New Zealand}
\email{byoungdu.kim@vuw.ac.nz}

\begin{document}

\begin{nouppercase}
\maketitle
\end{nouppercase}

\begin{abstract}
Let $A$ be an abelian variety defined over a number field $F$. Suppose its dual abelian variety $A'$ has good non-ordinary reduction at the primes above $p$. Let $F_{\infty}/F$ be a $\Zp$-extension, and for simplicity, assume that there is only one prime $\mfp$ of $F_{\infty}$ above $p$, and $F_{\infty, \mfp}/\Qp$ is totally ramified and abelian. (For example, we can take $F=\Q(\zeta_{p^N})$ for some $N$, and $F_{\infty}=\Q(\zeta_{p^{\infty}})$.)  As Perrin-Riou did in \cite{Perrin-Riou-1}, we use Fontaine's theory (\cite{Fontaine}) of group schemes to construct series of points over each $F_{n, \mfp}$ which satisfy norm relations associated to the Dieudonne module of $A'$ (in the case of elliptic curves, simply the Euler factor at $\mfp$), and use these points to construct characteristic power series $\bfL_{\alpha} \in \Qp[[X]]$ analogous to Mazur's characteristic polynomials in the case of good ordinary reduction. By studying $\bfL_{\alpha}$, we obtain a weak bound for $\rank E(F_n)$.

In the second part, we establish a more robust Iwasawa Theory for elliptic curves, and find a better bound for their ranks under the following conditions: Take an elliptic curve $E$ over a number field $F$. The conditions for $F$ and $F_{\infty}$ are the same as above. Also as above, we assume $E$ has supersingular reduction at $\mfp$. We discover that we can construct series of local points which satisfy finer norm relations under some conditions related to the logarithm of $E/F_{\mfp}$. Then, we apply Sprung's (\cite{Sprung}) and Perrin-Riou's insights to construct \textit{integral} characteristic polynomials $\bfLalg^{\sharp}$ and $\bfLalg^{\flat}$. One of the consequences of this construction is that if $\bfLalg^{\sharp}$ and $\bfLalg^{\flat}$ are not divisible by a certain power of $p$, then $E(F_{\infty})$ has a finite rank modulo torsions.
\end{abstract}
\tableofcontents

\begin{section}{Introduction}
A good place to start our discussion is Mazur's influential work on the rational points of abelian varieties over towers of number fields (\cite{Mazur}). Suppose $A$ is an abelian variety over a number field $F$, $A$ has good ordinary reduction at every prime above $p$, and $F_{\infty}$ is a $\Zp$-extension of $F$ (i.e., $\Gal(F_{\infty}/F) \cong \Zp$). First, he established the Control Theorem for $\Sel_p(A[p^{\infty}]/F_n)$'s (meaning he showed that the natural map $\Sel_p(A[p^{\infty}]/F_n)\to \Sel_p(A[p^{\infty}]/F_{\infty})^{\Gal(F_{\infty}/F_n)}$ has bounded kernel and cokernel as $n$ varies), and second, he demonstrated the existence of the characteristic polynomial $f(F_{\infty}/F, A)$ of $\Sel_p(A[p^{\infty}]/F_{\infty})$. (Any attempt to reduce his immense work to two sentences should be resisted, and readers should understand that the author is only trying to describe how his work has influenced this paper.) 

It means that we can use powerful tools of Iwasawa Theory. For example, if $f(F_{\infty}/F, A)\not=0$ (which is true if $A(F_n)^{\chi_n}$ and the $\chi_n$-part of the Shafarevich-Tate group $\Sha(A/F_n)[p^{\infty}]^{\chi_n}$ are finite for any $n\geq 0$ and any character $\chi_n$ of $\Gal(F_n/F)$), then $A(F_{\infty})$ has a finite rank modulo torsions. (Torsions over $F_{\infty}$ are often finite.)

Regarding the rank of $A(F_{\infty})$, now we have a stronger result for elliptic curves over $\Q$ by Kato (\cite{Kato}). However, we want to emphasize that Mazur's work and Kato's work have different goals and strengths.

Can we establish a result analogous to Mazur's for abelian varieties with good \textit{non-ordinary} reduction at primes above $p$? (See Section~\ref{Reduction} for the discussion about reduction types. We will not treat bad reduction primes, which seem to require a very different approach except for multiplicative reduction primes.)

The answer is that it is not easy to do Mazur's work directly for non-ordinary reduction primes. The main problem seems to be that the local universal norms are trivial when the primes are non-ordinary.

One of the more successful strategies to overcome this difficulty is to construct a series of local points which satisfy certain norm relations associated with the Euler factor $X^2-a_p(E)X+p$. Rubin introduced the idea of $\pm$-Selmer groups of elliptic curves (\cite{Rubin}). His method was to use the Heegner points as local points. Perrin-Riou (\cite{Perrin-Riou-1}) invented a way to construct such local points purely locally using Fontaine's theory of formal group schemes (\cite{Fontaine}). Her brilliant idea was all but forgotten for a long time, but are getting more influential recently. (And, this paper owes much to her work.)

More recently, Kobayashi (\cite{Kobayashi}) also constructed such local points of elliptic curves using a more explicit method, and demonstrated the potential that the theory for supersingular reduction primes can be as good as the theory for ordinary reduction primes.

Kobayashi assumed $a_p(E)=0$ for an elliptic curve $E$ defined over $\Q$ (which is automatically true by the Hasse inequality if $E$ has good supersingular reduction at $p$ and $p>3$). Sprung introduced a new idea, what he calls $\sharp/\flat$-Selmer groups for elliptic curves, which does not require $a_p(E)=0$ (\cite{Sprung}). His work has particular relevance to this paper because we are interested in the abelian varieties and elliptic curves over ramified fields. Even when we assume $a_p(E)=0$ or an equivalent condition, the assocaited formal groups behave as if $a_p$ is not 0 because the fields are ramified. We will make much use of his idea of the $\sharp/\flat$-decomposition in the second part.

Whereas our predecessors were concerned with abelian varieties over $\Q$ (and therefore formal groups defined over $\Qp$), we are concerned with abelian varieties defined over fields whose primes above $p$ are ramified, which present new difficulties.

First (Section~\ref{Case 1}), we take an abelian variety $A$ over a number field $F$, and let $A'$ be its dual abelian variety. For simplicity, we assume there is only one prime $\mfp$ of $F$ above $p$, and it is totally ramified over $F/\Q$. We assume $A'$ has good reduction at $\mfp$. Suppose $F_{\infty}$ is a $\Zp$-extension of $F$ such that $\mfp$ is totally ramified over $F_{\infty}/F$, and $F_{\infty, \mfp}/\Qp$ is abelian. For example, take $F =\Q(\zeta_{p^N})$ for some $N$, and $F_{\infty}=\Q (\zeta_{p^{\infty}})$.

Suppose $A'/F_{\mfp}$ has dimension $1$. (Generalizing to higher dimensions may not be very hard.) Let $H^{\vee}(X)=X^d+pb_1X^{d-1}+p^2b_2X^{d-2}+\cdots+p^db_d$ be the characteristic polynomial of the Verschiebung $\bfV$ acting on the Dieudonne module. For example, for an elliptic curve, that is simply $X^2-a_p(E)X+p$. Suppose that $A'(F_{\infty, \mfp})_{tor}$ is annihilated by some $M'>0$. Then, we construct points $Q(\pi_{N+n}) \in A'(F_{n, \mfp})$ such that we have

$$ \Tr_{F_{n, \mathfrak p}/F_{n-d, \mathfrak p}} Q(\pi_{N+n}) = \sum_{i=1}^d -p^i\cdot b_i \Tr_{F_{n-i, \mathfrak p}/F_{n-d, \mathfrak p}} Q(\pi_{N+n-i}).
$$

Fontaine's theory of finite group schemes (\cite{Fontaine}) is instrumental in our construction, as it is in Perrin-Riou's work (\cite{Perrin-Riou-1}). As Perrin-Riou does, for each root $\alpha$ of $H^{\vee}(X)$ with $v_p(\alpha)<1$, we can construct a characteristic power series $\bfL_{\alpha}(X)\in \Qp[[X]]$ which is analogous to Mazur's characteristic polynomial $f(F_{\infty}/F, A)$ except that it is not an integral power series unless $v_p(\alpha)=0$.

Then, we can obtain the following bound for the coranks of the Selmer groups (and thus, for the ranks of $A(F_n)$):

\begin{theorem}[Proposition~\ref{ZeroGo}]
Let $\lambda=v_p(\alpha)$.

\begin{enumerate}
\item If $\bfL_{\alpha}\not=0$, then
\[ \corank_{\Zp} \Sel_p(A[p^{\infty}]/F_n) \leq e(p-1) \times \left\{ p^{n-1}+p^{n-2}+ \cdots+ p^m \right\}+O(1)\]
where $n-m = \lambda n +O(1)$.

\item
If any root $\alpha$ of $H^{\vee}(X)$ has valuation 0 (i.e., if $A'$ has ``in-between'' reduction or ordinary reduction), then $\corank_{\Zp}(\Sel_p(A[p^{\infty}]/F_n))$ is bounded by the number of roots of $\bfL_{\alpha}$.
\end{enumerate}
\end{theorem}

We have $\bfL_{\alpha}\not=0$ if $\Sel_p(A[p^{\infty}]/F_n)^{\chi_n}$ is finite for any $n$ and any character $\chi_n$ of $\Gal(F_n/F)$. Also note that $\rank(A(F_n))$ is bounded by $\corank_{\Zp} \Sel_p(A[p^{\infty}]/F_n)$.

In addition, we construct similar local points over the extensions $F_{\mfp}(\sqrt[p^n]{\pi})$ ($n\geq 0$) for any uniformizer $\pi$ of $F_{\mfp}$ (Section~\ref{Kummer}). On one hand, this construction is fully general. On the other hand, since $\cup_n F_{\mfp}(\sqrt[p^n]{\pi})$ is not abelian over $F_{\mfp}$, it is not clear what we can do with it. (For instance, we cannot apply Iwasawa Theory to the points.)

Furthermore, assuming additional hypotheses, and with the crucial help of Sprung's insight, we can establish an Iwasawa Theory that is more closely aligned with Mazur's theory. In Section~\ref{Case 2}, we take an elliptic curve $E$ over $F$, and suppose $E$ has \textit{good supersingular reduction} at $\mfp$ (i.e., $a_{\mfp}(E)$ is not prime to $p$).

We choose a logarithm $\bfl$ of $E$ over $F_{\mfp}$ and a generator $\bfm$ of the Dieudonne module of $E$, and write

\[ \bfl=\alpha_1 \bfm+\alpha_2 \bfF\bfm\]
for some $\alpha_1, \alpha_2 \in F_{\mfp}$. We assume $p| \frac{\alpha_2}{\alpha_1}$ (Assumption~\ref{Assumption K}). Also we assume Assumption~\ref{Assumption L}, which is too technical to explain here, but is probably true in most cases.

One crucial step is that we modify our construction so that the resulting local points satisfy a finer norm relation (Proposition~\ref{Mark IV}). Another crucial step is that like Perrin-Riou, we construct $p$-adic characteristics, but this time, by applying an idea inspired by Sprung's insight of $\sharp/\flat$ (\cite{Sprung}), we construct integral $p$-adic characteristic polynomials $\bfLalg^{\sharp}(E), \bfLalg^{\flat}(E) \in \Lambda$. Since these are integral, they are more analogous to Mazur's characteristic $f(F_{\infty}/F, A)$, and it is likely that they have nice properties. They may not necessarily satisfy a control theorem in a literal sense, but nonetheless we manage to prove Proposition~\ref{DDT}, by which we can obtain the following.

\begin{theorem}[Theorem~\ref{DDR}]
Suppose $a_p$ and $\alpha$ are divisible by $p^T$ for some $T$, and neither $\bfLalg^{\sharp}(E)$ nor $\bfLalg^{\flat}(E)$ is divisible by $p^S$ for some $S$ with $S+\frac{[F:\Q]\times p}{(p-1)^2}<T$. Then, $E(F_{\infty})$ has a finite rank modulo torsions, and $\Sha(E/F_n)[p^{\infty}]^{\chi_n}$ is finite for all sufficiently large $n$ and primitive characters $\chi_n$ of $\Gal(F_n/F)$.
\end{theorem}

\end{section}

\begin{section}{Reduction Types}		\Label{Reduction}
In this short section, we discuss reduction types. 

For elliptic curves, what good reduction, good ordinary reduction, and good supersingular reduction mean is clear. Suppose an elliptic curve $E$ is defined over a local field $K$. Then, we may suppose it has a minimal model over $\OO_K$. Let $\tilde E$ denote the reduced curve of the minimal model modulo $\mm_{\OO_K}$. We say $E$ has good reduction if $\tilde E$ is non-singular (i.e., smooth). Furthermore, we say $E$ has good ordinary reduction if $\tilde E$ is non-singular, and $\tilde E[p]$ is non-trivial, and has good supersingular reduction if $\tilde E$ is non-singular, and $\tilde E[p]$ is trivial. There are other equivalent definitions.

For general abelian varieties, it may be advantageous to use the Dieudonne modules to define reduction types. (There are other definitions, but the one using Dieudonne modules seems relatively simple.) Suppose $G$ is a formal group scheme over $\OO_K$ where $K$ is a local field. Let $G_{/k}$ be its reduction over the residue field $k$. If $G_{/k}$ is smooth, then we say $G$ has good reduction. Assume $G$ has good reduction, and let $M$ be its Dieudonne module $\hat{CW}(R_{G_{/k}})$ where $R_{G_{/k}}$ is the affine algebra that defines $G_{/k}$, and $\hat{CW}$ denotes the completion of the co-Witt vectors. (See \cite{Dieudonne-1}, \cite{Dieudonne-2}, \cite{Fontaine}, or Section~\ref{Fontaine}.) The Frobenius $\bfF$ and the Verschiebung $\bfV$ act on $M$ through $\hat{CW}$ with $\bfF\bfV=\bfV\bfF=p$.

Let $H(X)$ be the characteristic polynomial of $\bfF$ as action on $M$, i.e., $H(X)=\det (X\cdot 1_M-\bfF|M)$. Write

\[ H(X)=X^d+a_{d-1}X^{d-1}+\cdots+a_0.\]
Then, $\bfF$ is a topological nilpotent if and only if the roots of $H(X)$ are non-units.

Since $\bfF\bfV=p$, 
$$H^{\vee}(X) \stackrel{def}=X^d+p\frac{a_1}{a_0}X^{d-1}+p^2 \frac{a_2}{a_0}X^{d-2}+\cdots+p^{d-1} \frac{a_{d-1}}{a_0} X+p^d \frac 1{a_0}$$
is the characteristic polynomial of $\bfV$ as action on $M$. We define the following terminology we will use in this paper.

\begin{definition} 		\Label{Calais}
Assume $G$ has good reduction. Also assume $\bfF$ is a topological nilpotent. Recall that  $\bfV$ is a topological nilpotent if all the roots of $H^{\vee}(X)$ are non-units.
\begin{enumerate}
\item If all the roots of $H^{\vee}$ are units, then we say $G$ has ordinary reduction.
\item If $\bfV$ is a topological nilpotent (i.e., all the roots of $H^{\vee}$ are non-untis), then we say $G$ has supersingular reduction.

\item If some roots of $H^{\vee}$ are units and some are not, then we say $G$ has in-between reduction.
\end{enumerate}
\end{definition}

The last terminology is our own ad-hoc invention.

Definition~\ref{Calais} makes it clear that in this paper, we assume $\bfF$ is a topological nilpotent, but this condition is used only in a minor way, and when we use that assumption, we will mention it.

\end{section}

\begin{section}{Fontaine's functor for ramified extensions}	\Label{Fontaine}

Our primary reference is \cite{Fontaine}~Chapter~4. We will keep his notation wherever possible. Fontaine's book is out of print, and not many libraries have a copy. So, we will explain his work briefly.

Let 

\begin{enumerate}[(a)]
\item $K'$: an extension over $\Qp$ (possibly ramified),
\item $\mathcal O_{K'}$: its ring of integers,
\item $\mathfrak m$: its maximal ideal,
\item $e$: the ramification index of $K'$.
\end{enumerate}

Let $k$ be the residue field of $\OO_{K'}$, and let $K$ be the fractional field of $W=W(k)$, the set of Witt vectors of $k$. In other words, it is the maximal unramified extension of $\Qp$ contained in $K'$. Then, there is the $p$-th Frobenius $\sigma$ on $K$. We let

\[ \mathbf D_k\stackrel{def}=W[\bfF, \bfV] \]
where 

\begin{enumerate}[(a)]
\item $\bfF$ acts $\sigma$-linearly, and $\bfV$ acts $\sigma^{-1}$-linearly on $W$. In other words, $\bfF a=\sigma(a)$ and $\bfV a=\sigma^{-1}(a)$ where $a\in W$.

\item $\bfF\bfV=\bfV\bfF=p$
\end{enumerate}
If $K'$ is totally ramified so that $k=\mathbb F_p$, we drop $k$ from $\mathbf D_k$.

Suppose $G$ is a smooth finite-dimensional (commutative) formal group scheme over $\OO_{K'}$ such that $G_{/k}$ is smooth. Fontaine found a way to describe $G$ by linear algebra. More specifically, he can describe $G$ completely up to isogeny (or, up to isomorphism if $e<p-1$) by the Dieudonne module $M$, and the set $L$ of its ``logarithms'', and his description is given by expressing the points of $G$ by the linear algebra of $L$ and $M$. Together, $(L, M)$ is called the Honda system of $G$.

We briefly summarize Fontaine's work: Let $R$ be the affine algebra of $G$ (i.e., $G(g)\cong \Hom(R, g)$ for any algebra $g$ over $\OO_{K'}$ where $\Hom$ is the set of ring homomorphisms). Then, $R_k=R/\mm R$ is the affine algebra of the special fiber $G_{/k}$. Set

\[ M\stackrel{def}=\Hom(G_{/k}, \hat{CW}) \]
where $\hat{CW}$ is the functor of completed co-Witt vectors. Then, $M=\Hom(G_{/k}, \hat{CW})\cong \hat{CW}(R_k)$. Since the Frobenius $\bfF$ and the Verschiebung $\bfV$ act on $CW$ by

\[ \bfF(\ldots, a_{-n},\ldots)=(\ldots, a_{-n}^p,\ldots),\]
\[ \bfV(\ldots, a_{-2}, a_{-1}, a_0)=(\ldots, a_{-2}, a_{-1}),\]
$\bfF$ and $\bfV$ also act on $M$ accordingly.
For any algebra $A$ over $k$,

\[ G_{/k}(A)\cong \Hom_{\mathbf D_k}(M, A). \]

Suppose $N$ is a $\Dieu_k$-module. Let $N^{(j)}$ denote the $\mathbf D_k$-module with the same underlying set $N$ and action twisted by $\sigma^j$. In other words, for $n \in N^{(j)}$ and $\lambda \in W$,

\[ \lambda\circ n=\sigma^{-j}(\lambda)n.\]

We note that $\bfF$ induces a $\mathbf D_k$-linear isomorphism $\bfF: M^{(j)} \to M^{(j-1)}$, and $\bfV$ induces a $\mathbf D_k$-linear isomorphism $\bfV: M^{(j)} \to M^{(j+1)}$. So, we can define the following maps:
\begin{enumerate}[(a)]
\item
$$\varphi_{i, j}:\mm^i\otimes_{\OO_{K'}} N^{(j)} \to \mm^{i-1}\otimes_{\OO_{K'}}N^{(j)}$$ 
is a natural map induced by the inclusion $\mathfrak m^i \to \mathfrak m^{i-1}$,

\item 
$$f_{i,j}: \mm^i \otimes_{\OO_{K'}} N^{(j)} \to \mm^i \otimes_{\OO_{K'}} N^{(j-1)}$$
induced by $\bfF: N^{(j)} \to N^{(j-1)}$, and

\item 
$$v_{i,j}: \mm^i \otimes_{\OO_{K'}} N^{(j)} \to \mm^{i-e} \otimes_{\OO_{K'}} N^{(j+1)}$$ 
given by $v_{i,j}(\lambda \otimes m)= p^{-1}\lambda \otimes \bfV m$.
\end{enumerate}

For a subset $I$ of $\Z\times \Z$, we let $\mathcal D_I(N)$ denote the system of diagrams (in the category of $\OO_{K'}$-modules) of the objects $\mm^i\otimes N^{(j)}$ where $(i,j) \in I$ and the maps $\varphi_{i,j}, f_{i,j}, v_{i,j}$ between the objects of $\mathcal D_I(N)$.  (See \cite{Fontaine}~p.189.)

We define

\[ I_0=\{ (i,j) \in \Z\times \Z, \; (j\geq 0) \; | \;  i\geq 0 \text{ if }j=0, i \geq p^{j-1}-je \text{ if }j \geq 1\}, \]
and let 

$$N_{\OO_{K'}}\stackrel{def}=\varinjlim\mathcal D_{I_0}(N).$$

For $j'>0$, we also define

\[ I_{j'}=\{ (i,j) \in \Z\times \Z, \;  (j \geq j') \; | i \geq p^{j-1}-je \}, \]
and let

$$N_{\mathcal O_{K'}}[j']\stackrel{def}=\varinjlim\mathcal D_{I_{j'}}(N).$$

When $M$ is a $\mathbf D_k$-module without $\bfF$-torsion, it is well-known that $M_{\mathcal O_{K'}}[1]\to M_{\mathcal O_{K'}}$ is injective, and

\[ M/\bfF M\cong   M_{\mathcal O_{K'}}/  M_{\mathcal O_{K'}}[1] \]
(\cite{Fontaine}~5.2.5, Corollaire~1).

\begin{definition}
\begin{enumerate}[(a)]
\item
For an algebra $g$ over $\OO_{K'}$, we can define

\begin{eqnarray*} \omega_g: \hat{CW}(g) &\to 	&	\Qp\otimes g 	\\
(\cdots, a_{-n},\cdots, a_{-1},a_0)	&\mapsto	&	\sum_{n=0}^{\infty} p^{-n}a_{-n}^{p^n}.
\end{eqnarray*}

\item
We define $P'(g)$ as the $\mathcal O_{K'}$-submodule of $\Qp\otimes g$ generated by $p^{-n} a^{p^n}$ for all $n\geq 0$ and all $a \in \mathfrak m \cdot g$.
\end{enumerate}

We will drop $g$ from $\omega_g$ if it does not cause confusion.
\end{definition}
 This group $P'(g)$ is not indefinitely large. In fact, we have

\[ \mathfrak m \cdot  g \subset P'(g) \subset \mathfrak m^{v} \cdot  g \]
where $v=\text{min}(p^n-ne)$ (in particular, if $e\leq p-1$, $P'(g)=\mathfrak m \cdot g$). See \cite{Fontaine}~p.197.

It is easy to see $\omega_g$ naturally extends to

\[ \omega'_g: \mathcal O_{K'} \otimes \hat{CW}_k(g/\mm \cdot g) \to \Qp \otimes g/P'(g)\]
by choosing a lifting $(\tilde a_{-n})\in \hat{CW}_k(g)$ of $(a_{-n})\in  \hat{CW}_k(g/\mm \cdot g)$.

\begin{proposition}[\cite{Fontaine}~Proposition~2.5]	\Label{Atlantic}
Let $N$ be a $\Dieu_k$-module so that $\bfV N=N$. Then, the canonical map $\OO_{K'}\otimes N \to N_{\OO_{K'}}$ is surjective, and its kernel is $\sum_{j=1}^{\infty} \mm^{p^{j-1}}\otimes \operatorname{Ker} \bfV^j$.

\end{proposition}

\begin{proposition}[\cite{Fontaine}~Lemme~3.1]	\Label{Pacific}
The kernel of $\omega_g'$ contains $\sum_{j=1}^{\infty} \mm^{p^{j-1}}\otimes \operatorname{Ker} V^j$.
\end{proposition}

There is a natural map $\mathcal O_{K'} \otimes \hat{CW}_k(g/\mm g) \to \hat{CW}_{k}(g/\mm g)_{\OO_{K'}}$. Note $\bfV   \hat{CW}_k(g/\mm g)= \hat{CW}_k(g/\mm g)$. Thus, by Propositions~\ref{Atlantic} and \ref{Pacific}, $\omega_g'$ factors through

\[ \omega_g: \hat{CW}_{k}(g/\mm g)_{\OO_{K'}} \to  \Qp \otimes g/P'(g) \]
(\cite{Fontaine}~p.197).

Recall that $R$ is the affine algebra of $G$. Then, there is the coproduct map $\delta: R\to R\hat\otimes_{\OO_{K'}} R$ which induces the group operation of $G$. 

Let $P_R$ be the $R$-module generated by $a^{p^n}/p^n$ for every $a \in R$ and $n\geq 0$. Let $L$ be the set of $a \in P_R$ so that $a\otimes 1 - \delta(a)+1\otimes a=0$. In other words, $L$ is the set of logarithms. It naturally satisfies
\[ L/\mathfrak m L \stackrel{\sim}\to M_{\mathcal O_{K'}}/  M_{\mathcal O_{K'}}[1] \cong  M/\bfF M. \]

Fontaine defined the following functor $G(L, M)$: 

\begin{definition}[\cite{Fontaine}~Section~4.4]	\Label{Fontaine-Nara}
For an algebra $g$ over $\OO_{K'}$ (i.e., $g$ is a ring containing $\OO_{K'}$), $G(L, M)(g)$ is the set of points $(\mathbf y, \mathbf x)$ with $\mathbf x \in \Hom_{\Dieudonne_k}( M, CW_k(g/\mm \cdot g))$, and $\mathbf y \in \Hom_{\mathcal O_{K'}}(L, \Qp\otimes g)$ satisfying the following: $\bfx$ naturally induces a map

$$\bfx_{\OO_{K'}}: M_{\mathcal O_{K'}} \to CW_k(g/\mm \cdot g)_{\mathcal O_{K'}}.$$
Then, $(\bfy, \bfx)$ is a fiber product in the sense that $\bfx_{\OO_{K'}}$ and $\bfy$ are identical through

\begin{eqnarray*}
\Hom_{\mathcal O_{K'} \otimes \Dieu_k}( M_{\mathcal O_{K'}}, CW_k(g/\mm g)_{\mathcal O_{K'}}) \to & \Hom_{\mathcal O_{K'}}( L,  \Qp\otimes g/ P'(g))\\
&\uparrow\\
&\Hom_{\mathcal O_{K'}}(L,  \Qp\otimes g),
\end{eqnarray*}
\end{definition}

There is a natural map $i_G:G\to G(L, M)$, and also we can find a map in the reverse direction $j_G:G(L, M)\to G$. These maps are not necessarily isomorphisms unless $e<p-1$. Rather, $i_G\circ j_G=p^t$, $j_G\circ i_G=p^t$ for some $t$ which depends on the ramification index $e$.
\end{section}

\begin{section}{Perrin-Riou's insight, and weak bounds for ranks}	\Label{Constructing}

In this section, we construct points of formal group schemes over local fields satisfying certain norm relations. The local points we construct are analogous to the points that Perrin-Riou constructed (\cite{Perrin-Riou-1}), and indeed, this section is an effort to find a way to make her idea work for group schemes defined over ramified fields. As in her work, Fontaine's functor (\cite{Fontaine}, and also \cite{Dieudonne-1}, \cite{Dieudonne-2}) plays a central role, but we need a functor defined for group schemes over ramified fields. There is a brief discussion about the functor in the previous section (Section~\ref{Fontaine}). And then, again following Perrin-Riou, we construct power series analogous to Mazur's characteristic polynomials of the Selmer groups. Our power series have limited utility unlike Mazur's characteristics because they are not integral. Nonetheless, they give a bound for the coranks of the Selmer groups (thus a bound for the ranks of the Mordell-Weil groups).

\begin{subsection}{Constructing the Perrin-Riou local points}		\Label{Some special}
Suppose $k_{\infty}/\Qp$ is a totally ramified normal extension with $\Gal(k_{\infty}/\Qp)\cong \Z_p^{\times}$. By local class field theory, it is given by a Lubin-Tate group of height $1$ over $\Zp$. In other words, there is $\varphi(X)=X^p+\alpha_{p-1} X^{p-1}+\cdots+\alpha_1X \in \Zp[X]$ with $p|\alpha_i$, $v_p(\alpha_1)=1$ so that

\[ k_{\infty}=\cup_n \Qp(\pi_n) \]
where $\varphi(\pi_n)=\pi_{n-1}$ ($\pi_n\not=0$ for $n > 0$, $\pi_0=0$).

\begin{remark}		\Label{Edward the Confessor}
We can also study a more general case where $k_{\infty}/\Qp$ is ``merely'' ramified (rather than totally ramified). It can certainly be done as the author did in a different context and for a different problem in \cite{Kim-1}. The notation will become much more complicated.
\end{remark}

Suppose $K'=\Qp(\pi_N)$ for some $N>0$. Let $\mm=\mm_{\OO_{K'}}$, and $k$ be $\OO_{K'}/\mm_{\OO_{K'}}$ (which is simply $\mathbb F_p$).

We let $G$ be a formal group scheme over $\mathcal O_{K'}$ such that its reduced group scheme $G_{/k}$ (i.e., the special fiber) is smooth (therefore, $G$ has good reduction).

As in section~\ref{Fontaine}, we set $M=\Hom(G_{/k}, \hat{CW})$, which is a $\mathbf D$-module, and define $L$ as we did in Section~\ref{Fontaine}. In addition, we assume 

\begin{assumption}
The dimension of $G$ is $1$ (i.e., $L$ is rank $1$ over $\OO_{K'}$).
\end{assumption}
This assumption will make our work much simpler. 

\begin{remark}
Even though the author has not thought much about it, the case where the dimension of $G$ is not $1$ may not be so difficult. We only need to consider multiple logarithms.
\end{remark} 
Also, we assume

\begin{assumption}	\Label{Assumption-1}
Recall that we assume $G$ has good reduction. Also we assume $G$ does not have ordinary reduction. (See Definition~\ref{Calais}.)
\end{assumption}
Clearly, the case where $G$ has good ordinary reduction is covered well by Mazur's work (\cite{Mazur}).

Since we always assume that $\bfF$ acts on $M$ as a topological nilpotent, $M$ can be considered as a $\Zp[[\bfF]]$-module.

We set

\[ d=\rank_{\Zp} M. \]
Since we assume $G_{/k}$ is of dimension $1$, $\dim_{\mathbb F_p} M/\bfF M=1$, thus we may choose $\bfm \in M$ so that it generates $M$ over $\Zp[[\bfF]]$. More specifically, 

\[ \bfm, \bfF\bfm, \cdots, \bfF^{d-1} \bfm,\]
are $\Zp$-linearly independent, and generate $M$ over $\Zp$.

\begin{remark}
In fact, this seems to be the only place in this section where we use the condition that $\bfF$ is a topological nilpotent.
\end{remark}

We may also choose an $\OO_{K'}$-generator $\bfl$ of $L$. Since $L \subset M_{\OO_{K'}}$, we may write

\begin{eqnarray*}
\bfl	&=&	(\bfl_{ij})_{(i,j)\in I_0},\\
\bfl_{ij}		&=&	\sum_{k=0}^{d-1} \alpha_k^{(ij)} \bfF^k \bfm \in 	\mm^i \otimes M^{(j)}
\end{eqnarray*}
for some $\alpha_k^{(ij)} \in \mm^i$.

We set

\[ H(X)={\det}_{\Zp} (X\cdot 1_M-\bfF|M)=X^d+a_{d-1}X^{d-1}+\cdots+a_0 \in \Zp[X], \]
then 

\[ \bar H(X)\stackrel{def}= \ds \frac{H(X)}{a_0} = 1+ \ds \frac{a_1}{a_0}X+\cdots+\frac{a_{d-1}}{a_0}X^{d-1}+\frac1{a_0}X^d. \]
We let

\[J(X)\stackrel{def}=\bar H(X) -1=b_1X+b_2X^{2}+\cdots+b_dX^{d} \]
then formally we have

\[ \bar H(X)^{-1}=1-J(X)+J(X)^2-\cdots .\]

\begin{notation}
\begin{enumerate}
\item Recall $H(X)=X^d+a_{d-1}X^{d-1}+\cdots+ a_0$, and $\varphi(X)=X^p+\alpha_{p-1} X^{p-1}+\cdots+\alpha_1X$. Define

\[ \epsilon\stackrel{def}=\ds \frac{a_0 \alpha_{p-1}}{p\cdot (a_0+a_1+\cdots+a_{d-1}+1)}. \]
(Note that $\epsilon \in p\Zp$ because $a_0+a_1+\cdots+a_{d-1}+1 \in 1+p\Zp$.)

\item Let $\mathcal P$ be the $\Zp[[X]]$-submodule of $\Qp[[X]]$ which is generated by $\frac{X^{p^n}}{p^n}$ for $n=0,1,2,\cdots$. And, let $\bar{\mathcal P}=\mathcal P/p\Zp[[X]]$, which is isomorphic to $\hat{CW}(\mathbb F_p[[X]])$ through 

\begin{eqnarray*}	\omega: \hat{CW}(\mathbb F_p[[X]]) 	&\to	& \bar{\mathcal P}		\\
(\cdots, a_1, a_0)		&\mapsto	& \sum \ds \frac{\tilde a_n^{p^n}}{p^n}
\end{eqnarray*}
($\tilde a_n \in \Zp[[X]]$ is a lifting of $a_n$).

\item
Let $\varphi$ be an operator on $\mathcal P$ given by

\[ \varphi(X^n):=\varphi(X)^n \]
which is equivalent to $\bfF$ on $\bar{\mathcal P} \cong \hat {CW}(\mathbb F_p)$.
(More precisely, for $a\in W$,

\[ \varphi(a)=\sigma(a) \]
where $\sigma$ is the $p$-th Frobenius map on $W$, and thus $\varphi$ is a $\sigma$-linear operator. But, here we have $W=\Zp$, so we can safely ignore this.)

Then, we define
\[ l(X)=\left[ 1-J(\varphi)+J(\varphi)^2-\cdots \right] \circ X. \]

\item  Define $\bfx \in G_{/k}(\mathbb F_p[[X]]) \cong \Hom_{\Dieu}(M, \hat {CW}(\mathbb F_p[[X]])) \cong \Hom_{\Dieu}(M, \bar{\mathcal P})$ by

\[ \bfx(\bfm)= l(X) \pmod{p\Zp[[X]]}\]
and extend $\Dieu$-linearly. (Note that $\Hom_{\Dieu}(M, \bar{\mathcal P}) \cong \Hom_{\Zp[\bfF]}(M, \bar{\mathcal P})$ by \cite{Perrin-Riou-1}~Section~3.1~p.261.)
\end{enumerate}
\end{notation}

\begin{proposition}		\Label{Sino-Japanese War}
$\bfx$ is well-defined.
\end{proposition}
\begin{proof}
First, we need to show $l(X)$ is well-defined. Because $G$ has supersingular reduction, $p^i b_i \in p\Zp$ ($i=1,2,\cdots, d$). Thus, $l(X)$ is well-defined (i.e., the infinite summation which defines $l(X)$ is convergent). Then, we check
\[ (1+J(\varphi))\circ \left\{ 1-J(\varphi)+J(\varphi)^2-\cdots \right\} \circ X=1\circ X.\]
Since $\bfF$ is a topological nilpotent, $p|a_0$, thus $H(\varphi)\circ l(X)=a_0X \in p\Zp[[X]]$, in other words, $H(\bfF)\circ l(X)=0 \in \hat CW(\mathbb F_p [[X]])$. Since $H(X)$ is irreducible, $\bfx$ extends to the entire $M$ $\Dieu$-linearly.
\end{proof}

\begin{notation}		\Label{Moscow}
\begin{enumerate}
\item Define a lifting $\tilde \bfx \in \Hom_{\Zp}(M, \mathcal P)$ of $\bfx$ by

\[ \tilde \bfx(\bfF^k \bfm)=\epsilon+\varphi^k \circ l(X)=\epsilon+ l(\varphi^{(k)}(X)), \quad k=0,1,\cdots,d-1\]
where $\varphi^{(k)}=\varphi(\varphi(\cdots(X)))$ ($k$-times).

\item Recall

\begin{eqnarray*}
\bfl	&=&	(\bfl_{ij})_{(i,j)\in I_0},\\
\bfl_{ij}		&=&	\sum_{k=0}^{d-1} \alpha_k^{(ij)} \bfF^k \bfm \in 	\mm^i \otimes M^{(j)}.
\end{eqnarray*} 
We can write

\[	\bfF^j \bfl_{ij}		=	\sum_{k=0}^{d-1} \beta_k^{(ij)} \bfF^k \bfm\]
for some $\beta_k^{(ij)}\in \mm^i$.

\item
Define $\bfy \in \Hom_{\OO_{K'}}(L, K'[[X]])$ explicitly as follows:

We set

\[ \bfy(\bfl)=\sum_{(i,j)\in I_0} \sum_{k=0}^{d-1} \beta_k^{(ij)} \tilde\bfx (\bfF^k \bfm) \]
and extend to $L$ $\OO_{K'}$-linearly.

\item Then, we set $P=(\bfy, \bfx) \in G(L, M)(\OO_{K'}[[X]])$.
\end{enumerate}
\end{notation}

\begin{proposition}
$P=(\bfy, \bfx)$ is well-defined.
\end{proposition}
\begin{proof}
We need to show it is a fiber product in the sense of Definition~\ref{Fontaine-Nara}. We let $\bfx$ also denote the extended map $\bfx: M_{\OO_{K'}}\to \hat{CW}(\mathbb F_p[[X]])_{\OO_{K'}}$.

For each $\bfl_{ij}\in \mm^i \otimes M^{(j)}$, 

$$\bfx(\bfl_{ij})=\bfx(\sum_{k=0}^{d-1} \alpha_k^{(ij)} \bfF^k \bfm)=\sum_{k=0}^{d-1} \alpha_k^{(ij)} \bfx(\bfF^k \bfm) \in \mm^i \otimes \hat{CW}(\mathbb F_p[[X]])^{(j)}.$$

Because $\omega$ on $\hat{CW}(\mathbb F_p[[X]])_{\OO_{K'}}$is deduced from $\omega: \OO_{K'}\otimes \hat{CW}(\mathbb F_p[[X]]) \to K'[[X]]/P'(\OO_{K'}[[X]])$ through $\OO_{K'}\otimes \hat{CW}(\mathbb F_p[[X]]) \to \hat{CW}(\mathbb F_p[[X]])_{\OO_{K'}}$, to evaluate $\omega$ on $\sum_{k=0}^{d-1} \alpha_k^{(ij)} \bfx(\bfF^k \bfm) \in \mm^i \otimes \hat{CW}(\mathbb F_p[[X]])^{(j)}$, we need to send it to $p^j \cdot \mm^i \otimes \hat{CW}(\mathbb F_p[[X]])$ by $\bfF^j$, and obtain

\begin{eqnarray*}
\omega(\bfx(\bfl_{ij}))	&=&		\omega \left( \bfF^j\sum_{k=0}^{d-1} \alpha_k^{(ij)}  \bfx(\bfF^k \bfm) \right)	\\
&=& \omega \left( \sum_{k=0}^{d-1} \beta_k^{(ij)}  \bfx(\bfF^k \bfm) \right)		\\
&=& \sum_{k=0}^{d-1} \beta_k^{(ij)} l(\varphi^k(X)) \pmod{P'(\OO_{K'}[[X]])}.
\end{eqnarray*}
Thus, $\omega(\bfx(\bfl))=\bfy(\bfl) \pmod{P'(\OO_{K'}[[X]])}$, and by extending $\OO_{K'}$-linearly, $\bfx=\bfy$ as elements of $\Hom_{\OO_{K'}}(L, K'[[X]]/P'(\OO_{K'}[[X]]))$, and our claim follows.
\end{proof}

For simplicity, let $\Tr_{n/m}$ denote $\Tr_{K'(\pi_n)/K'(\pi_m)}$.

\begin{proposition}		\Label{Despicable-Laundry-Machine}
Modulo torsions, we have

\begin{eqnarray*} \Tr_{n/n-d} P(\pi_n) &=& -p\cdot b_1\cdot \Tr_{n-1/n-d} P(\pi_{n-1} ) - p^2 \cdot b_2 \cdot \Tr_{n-2/n-d} P(\pi_{n-2} )		\\
&& - \quad  \cdots \quad - p^d \cdot b_d\cdot P(\pi_{n-d} )
\end{eqnarray*}
for every $n\geq N+d$.
\end{proposition}

\begin{proof}
Note that $(0, \mathbf z) \in G(L, M)(g)$ is a torsion point for any $\mathbf z \in G_{/k}(g/\mm g)$. Thus, we only need to show the identity of the $L$-parts.

First, we find

\begin{eqnarray*} \Tr_{n/n-1} l(\pi_n) &=& \Tr_{n/n-1}  \left. \left[ 1-J(\varphi)+J(\varphi)^2-\cdots \right] \circ X \right|_{X=\pi_n} \\
&=&\Tr_{n/n-1} \pi_n - \Tr_{n/n-1} J(\varphi)\circ \left. \left[ 1-J(\varphi)+J(\varphi)^2- \right]\circ X \right|_{X=\pi_n} \\
&=& -\alpha_{p-1}-\Tr_{n/n-1} \left[ b_1l(\varphi(X))+\cdots+b_{d}l(\varphi^{(d)}(X))\right]_{X=\pi_n}	\\
&=& -\alpha_{p-1}-p\cdot \left[b_1 l(\pi_{n-1})+\cdots+b_d l(\pi_{n-d}) \right].
\end{eqnarray*}
Then, we can also find

\begin{eqnarray*} \Tr_{n/n-d} l(\pi_n) &=& -p^{d-1}\alpha_{p-1}-p\cdot b_1 \cdot \Tr_{n-1/n-d} l(\pi_{n-1})	\\
&& -\quad \cdots \quad -p^{d-1}\cdot b_{d-1} \cdot  \Tr_{n-d+1/n-d} l(\pi_{n-d+1})-p^d\cdot b_d  \cdot l(\pi_{n-d}).
\end{eqnarray*}

We recall that $b_1=\frac{a_1}{a_0},\cdots, b_{d-1}=\frac{a_{d-1}}{a_0}, b_d=\frac 1{a_0}$, thus from the definition of $\epsilon$, we have

\[ p^d\cdot \left( 1+\ds \frac{a_1}{a_0}+\cdots+\frac{a_{d-1}}{a_0}+\frac1{a_0} \right) \cdot \epsilon=p^{d-1} \cdot \alpha_{p-1}. \]
Thus, we have

\begin{multline} 		\Label{Tokyo}
\Tr_{n/n-d}(\epsilon+l(\pi_n)) =-p\cdot \ds \frac{a_1}{a_0} \cdot \Tr_{n-1/n-d} (\epsilon+l(\pi_{n-1}))	\\
\qquad \qquad \qquad - \quad \cdots  \quad - p^{d-1}\cdot \ds \frac{a_{d-1}}{a_0} \cdot  \Tr_{n-d+1/n-d} (\epsilon+ l(\pi_{n-d+1}))	-p^d\cdot \frac1{a_0}  \cdot (\epsilon+ l(\pi_{n-d})).
\end{multline}
Similarly, we check the following: For $0<i<d$,

\begin{eqnarray*} (\varphi^i\circ l)(\pi_n) &=& \varphi^{(i)}(\pi_n) -\varphi^i\circ J(\varphi)\circ [1-J(\varphi)+J(\varphi)^2-\cdots ] \circ X |_{\pi_n}		\\
&=& \pi_{n-i} - \left[ b_1 l(\pi_{n-i-1})+\cdots+b_d l(\pi_{n-i-d}) \right].
\end{eqnarray*}
Then, we have 

\begin{eqnarray*} \Tr_{n/n-d} (\varphi^i\circ l) (\pi_n) &=& -p^{d-1}\alpha_{p-1}-p\cdot b_1 \cdot \Tr_{n-1/n-d} l(\pi_{n-i-1})	\\
&& -\quad \cdots \quad -p^{d-1}\cdot b_{d-1} \cdot  \Tr_{n-d+1/n-d} l(\pi_{n-i-d+1})-p^d\cdot b_d  \cdot l(\pi_{n-i-d})		\\
&=& -p^{d-1}\alpha_{p-1}-p\cdot b_1 \cdot \Tr_{n-1/n-d} (\varphi^i\circ l)(\pi_{n-1})	\\
&& -\quad \cdots \quad -p^{d-1}\cdot b_{d-1} \cdot  \Tr_{n-d+1/n-d} (\varphi^i\circ l) (\pi_{n-d+1})	\\
&&-p^d\cdot b_d  \cdot (\varphi^i\circ l) (\pi_{n-d}),
\end{eqnarray*}
and by repeating the argument used above, we obtain an identity analogous to (\ref{Tokyo}). 

Recall

\[ \bfy(\bfl)=\sum_{(i,j)\in I_0} \sum_{k=0}^{d-1} \beta_k^{(ij)} (\epsilon+l(\varphi^{(k)}(X)) \]
from Notation~\ref{Moscow}. By the above discussion, we have

\begin{eqnarray*} \Tr_{n/n-d} \bfy(\bfl)|_{X=\pi_n} &=& -p\cdot b_1\cdot \Tr_{n-1/n-d} \bfy(\bfl) |_{X=\pi_{n-1}} - p^2 \cdot b_2 \cdot \Tr_{n-2/n-d} \bfy(\bfl) |_{X=\pi_{n-2}}		\\
&& - \quad  \cdots \quad - p^d \cdot b_d\cdot \bfy(\bfl) |_{X=\pi_{n-d}}
\end{eqnarray*}
and by extending it to $L$ $\OO_{K'}$-linearly, we obtain our claim.
\end{proof}

\end{subsection}

\bigskip

\begin{subsection}{Construction for Kummer Extensions}		\Label{Kummer}

\hfill

Now we define a slightly different operator $\varphi$ on $\Zp[[X]]$ by 

\[ \varphi(X)=X^p, \quad \varphi(a)=\sigma(a), a \in \Zp \]
where $\sigma$ is the $p$-th Frobenius map mentioned earlier (actually, $\sigma$ acts trivially on $\Zp$, so the action of $\varphi$ on $\Zp$ is purely symbolic.)

\begin{notation}
\begin{enumerate}
\item $K'$ is a totally ramified extension of $\Qp$, and $\zeta_p \not\in K'$. Let $\mm$ denote $\mm_{\OO_{K'}}$.

\item Set $e=[K':\Qp]$.	Assume $e<p$.

\item
Choose a uniformizer $\pi$ of $K'$, and choose $\pi_n$ for every $n \geq 0$ such that

\[ \pi_0=\pi,\quad \pi_{n+1}^p=\pi_n \quad \text{  for every  }\quad n \geq 0.\] 

\item For any $n\geq m\geq 0$, we let $\Tr_{n/m}$ denote $\Tr_{K'(\pi_n)/K'(\pi_m)}$.
\end{enumerate}
\end{notation}

Supppose $G$ is a formal group scheme of dimension $1$ over $\OO_{K'}$, its reduced scheme $G_{/k}$ over $k=\OO_{K'}/\mm$ is smooth (thus $G$ has good reduction), and $G$ has supersingular reduction. We recall from Section~\ref{Fontaine} that a Honda system $(M, L)$ is attached to $G$.

Like Section~\ref{Some special}, we choose an $\OO_{K'}$-generator $\bfl$ of $L$ and a $\Zp[\bfF]$-generator $\bfm$ of $M$. Then,

\begin{eqnarray*}
\bfl	=	(\bfl_{ij})_{(i,j) \in I_0}, \quad \bfl_{ij}= \sum_{k=0}^{d-1} \alpha_k^{(ij)} \bfF^k \bfm \in \mm^i \otimes M^{(j)}
\end{eqnarray*}
for some $\alpha_k^{(ij)} \in K'$.

Again, similar to Section~\ref{Some special}, we define

\begin{notation}
\[ H(X)={\det}_{\Zp} (X\cdot 1_M-\bfF|M)=X^d+a_{d-1}X^{d-1}+\cdots+a_0 \in \Zp[X], \]

\[ \bar H(X)\stackrel{def}= \ds \frac{H(X)}{a_0},\]

\[ J(X)\stackrel{def}=\bar H(X) -1=b_1X+b_2X^{2}+\cdots+b_dX^{d} \]

\[ l(X)\stackrel{def}= \left\{ 1-J(\varphi)+J(\varphi)^2-\cdots \right\} \circ X. \]
\end{notation}
\bigskip

\begin{proposition}
Recall $G(k[[X]])\cong \Hom_{\Zp[\bfF]}(M, \bar {\mathcal P}$) (\cite{Perrin-Riou-1}~Section~3.1 p.261). We define $\bfx \in G(k[[X]])$ by

\[ \bfx(\bfm)=l(X) \pmod{p\Zp[[X]]},\]
and expand $\Zp[\bfF]$-linearly. Then, $\bfx$ is well-defined.
\end{proposition}
\begin{proof}
See Proposition~\ref{Sino-Japanese War}.
\end{proof}

Now, we choose a lifting $\bfy \in \Hom_{\OO_{K'}}(L,  K'[[X]])$ of $\bfx$ as follows:

\begin{notation}

\begin{enumerate}
\item Define a lifting $\tilde \bfx \in \Hom_{\Zp}(M, \mathcal P)$ of $\bfx$ by

\[ \tilde \bfx(\bfF^i \bfm)=\varphi^i \circ l(X)=l(X^{p^i}), \quad i=0,1,\cdots,d-1.\]
Then, define $\bfy \in \Hom_{\OO_{K'}}(L, K'[[X]])$ explicitly as follows:

Write $\bfF^j \bfl_{ij}=\sum_{k=0}^{d-1} \beta_k^{(ij)} \bfF^k \bfm$ for some $\beta_k^{(ij)} \in K'$. We set

\[ \bfy(\bfl)=\sum_{(i,j) \in I_0} \sum_{k=0}^{d-1} \beta_k^{(ij)} \tilde\bfx (\bfF^k\bfm)= \sum_{(i,j) \in I_0} \sum_{k=0}^{d-1} \beta_k^{(ij)} l(X^{p^k}) \]
and expand $\bfy$ $\OO_{K'}$-linearly.

\item Then, we set $P=(\bfx, \bfy) \in G(M,L)(\Zp[[X]] \otimes \OO_{K'})$.

\end{enumerate}
\end{notation}

We note

\begin{eqnarray}	\Label{Note}
\Tr_{n/n-1} \pi_n^i =0 \quad \text{for all }n >0
\end{eqnarray}
for $i\leq e$ because $e<p$.

\begin{proposition} 	\Label{Chicken-Burger}
For $n>d$ and $i=1,2,\cdots, e$, modulo torsions, we have

\begin{eqnarray*} \Tr_{n/n-d} P(\pi_n^i) &=& -p\cdot b_1\cdot \Tr_{n-1/n-d} P(\pi_{n-1}^i) - p^2 \cdot b_2 \cdot \Tr_{n-2/n-d} P(\pi_{n-2}^i)		\\
&& - \quad  \cdots \quad -b_d\cdot p^d \cdot P(\pi_{n-d}^i).
\end{eqnarray*}
\end{proposition}

\begin{proof}
This is similar to Proposition~\ref{Despicable-Laundry-Machine} in Section~\ref{Some special}, so we will provide only a brief proof. 

\begin{eqnarray*} \Tr_{n/n-1}l(\pi_n^i)&=& \Tr_{n/n-1}\left\{ \pi_n^i -[J(\varphi)-J(\varphi)^2+\cdots]\circ X|_{X=\pi_n^i} \right\} \\
&=& \Tr_{n/n-1} \left\{-[J(\varphi)-J(\varphi)^2+\cdots]\circ X|_{X=\pi_n^i}  	\right\}	\\
&=& -p\cdot (J(\varphi)\circ l)(\pi_n^i).
\end{eqnarray*}
The last line is equal to

\begin{eqnarray*} -p\cdot (J(\varphi)\circ l)(\pi_n^i)&=& -p \left\{ b_1 \cdot l(\pi_n^{i\cdot p})+b_2 \cdot  l(\pi_n^{i\cdot p^2})+\cdots+b_d  \cdot l(\pi_n^{i\cdot p^d}) \right\}		\\
&=& -p \left\{ b_1 \cdot  l(\pi_{n-1}^i) + b_2 \cdot  l(\pi_{n-2} ^i)+  \cdots  + b_d  \cdot l(\pi_{n-d}^i) \right\}.
\end{eqnarray*}
Thus by applying $\Tr_{n-1/n-d}$ to it, we have

\begin{eqnarray*} \Tr_{n/n-d} l(\pi_n^i)	&=& -p \cdot b_1 \cdot  \Tr_{n-1/n-d} l(\pi_{n-1}^i)-p^2 \cdot  b_2 \cdot  \Tr_{n-2/n-d}l(\pi_{n-2}^i)		\\
&&-\quad \cdots\quad -p^d  \cdot b_d \cdot  l(\pi_{n-d}^i).
\end{eqnarray*}

Also similar to Proposition~\ref{Despicable-Laundry-Machine}, for $j=1,\cdots, d-1$ we have

\begin{eqnarray*} \Tr_{n/n-d} l((\pi_n^i)^{p^j}) 
&=& -b_1 \cdot p \cdot \Tr_{n-1/n-d} l((\pi_{n-1}^i)^{p^j} )- b_2 \cdot p^2\cdot \Tr_{n-2/n-d} l((\pi_{n-2}^i)^{p^j}) 	\\
&&- \quad \cdots \quad - b_d \cdot p^d \cdot l((\pi_{n-d}^i)^{p^j})	.
\end{eqnarray*}

Thus, we have

\begin{eqnarray*} \Tr_{n/n-d} \bfy (\bfl)|_{X=\pi_n^i} &=& -p \cdot b_1 \cdot \Tr_{n-1/n-d} \bfy (\bfl)|_{X=\pi_{n-1}^i}		\\
&&-p^2 \cdot b_2 \cdot \Tr_{n-2/n-d} \bfy (\bfl)|_{X=\pi_{n-2}^i}		\\
&&-\quad \cdots\quad 		\\
&&-p^d  \cdot b_d \cdot \bfy (\bfl)|_{X=\pi_{n-d}^i}.
\end{eqnarray*}

Similar to Proposition~\ref{Despicable-Laundry-Machine}, we obtain our claim.
\end{proof}

The problem is that we do not know whether these points are useful or not. The extension $K'(\pi_{\infty})/K'$ is not even normal. Its normal closure $K'(\pi_{\infty}, \zeta_{p^{\infty}})/K'$ is not abelian. So, it seems impossible to use Iwasawa Theory, and the author cannot see any other use for them.

\end{subsection}

\begin{subsection}{The Perrin-Riou characteristics, and weak bounds for ranks}		\Label{Case 1}

\hfill

In this section, we apply the construction in Section~\ref{Some special}. As in that section, we suppose $k_{\infty}/\Qp$ is a totally ramified normal extension with $\Gal(k_{\infty}/\Qp)\cong \Z_p^{\times}$. By local class field theory, it is given by a Lubin-Tate group of height $1$ over $\Zp$. In other words, there is $\varphi(X)=X^p+\alpha_{p-1} X^{p-1}+\cdots+\alpha_1X \in \Zp[X]$ with $p|\alpha_i$, $v_p(\alpha_1)=1$ so that

\[ k_{\infty}=\cup_n \Qp(\pi_n) \]
where $\varphi(\pi_n)=\pi_{n-1}$ ($\pi_n\not=0$ for $n > 0$, $\pi_0=0$).

We let $F$ be a number field, $F_{\infty}$ be a $\Zp$-extension of $F$ (i.e., $\Gal(F_{\infty}/F) \cong \Zp$), $A$ be an abelian variety over $F$, and $A'$ be its dual abelian variety over $F$ so that there is a non-degenerate Weil pairing $e_n: A[n]\times A'[n]\to \Z/n\Z$ for every integer $n$, which is non-degenerate and commutative with the action of $G_F$. Let $\bfT=T_pA$, and let $\bfA\stackrel{def}=\varinjlim \bfT/p^n \bfT$.

In this section, we suppose there is only one prime $\mathfrak p$ of $F$ above $p$, $\mathfrak p$ is totally ramified over $F/\Q$, $\mathfrak p$ is totally ramified over $F_{\infty}/F$, $F_{\infty, \mathfrak p}=k_{\infty}$, and $F_{\mathfrak p}=\Qp(\pi_N)$ for some $N\geq 1$.

Let $G/\OO_{F_{\mathfrak p}}$ denote the formal completion of $A'/F_{\mathfrak p}$. As in Section~\ref{Some special}, we assume $G$ has dimension $1$, which means that the group of its logarithms has rank $1$ over $\OO_{F_{\mathfrak p}}$.

\begin{example}
An obvious example that satisfies all these conditions is an elliptic curve $E$ defined over $\rat(\zeta_{p^N})$ with good supersingular reduction at the unique prime $\mathfrak p$ above $p$.
\end{example}

We recall the points $P(\pi_n) \in G(M, L)(\mm_{\Qp(\pi_n)})$ constructed in Section~\ref{Some special}.

\begin{assumption} There is $M'>0$ so that $M' \cdot G(\OO_{k_{\infty}})_{tors}=0$.
\end{assumption}

This assumption is obviously true if $G(\OO_{k_{\infty}})_{tors}$ is finite.

\begin{definition}

\begin{enumerate}[(a)]
\item
Let $M$ be the Dieudonne module $\Hom(G_{/\mathbb F_p}, \hat{CW})$, and $L$ be the set of logarithms of $G$ as defined in Section~\ref{Fontaine}.

\item As in Section~\ref{Some special}, we set
\[ H(X)={\det}_{\Zp} (X\cdot 1_M-\bfF|M)=X^d+a_{d-1}X^{d-1}+\cdots+a_0 \in \Zp[X], \]
and

\begin{eqnarray*} \bar H(X) &\stackrel{def}=& \ds \frac{H(X)}{a_0} = 1+ \ds \frac{a_1}{a_0}X+\cdots+\frac{a_{d-1}}{a_0}X^{d-1}+\frac1{a_0}X^d 		\\
&=&1+b_1X+b_2X^{2}+\cdots+b_dX^{d}.
\end{eqnarray*}

\end{enumerate}
\end{definition}

\begin{definition} From Section~\ref{Fontaine}, recall that there is a natural map $i_G:G\to G(L, M)$, and a map $j_G:G(L, M)\to G$ so that $i_G\circ j_G=p^t$, $j_G\circ i_G=p^t$ for some $t$ which depends on the ramification index $e$. Also, let $i':G\to A'$ be the natural injection from the formal group scheme $G$ to the abelian variety $A'$. We define

\begin{enumerate}[(a)]
\item Where $e=[\Qp(\pi_N):\Qp]=[F_{\mfp}:\Qp]$, let 

$$\{ \pi_{N,1},\cdots, \pi_{N, e}\}=\{ \pi_N^{\sigma} \}_{\sigma \in\Gal(\Qp(\pi_N)/\Qp)}.$$

\item Then, for every $n> N$ and for each $i=1,\cdots,e$, choose $\pi_{n,i}$ so that $\varphi(\pi_{n,i})=\pi_{n-1,i}$.

\item	For $i=1,\cdots, e$,
\[ Q(\pi_{N+n, i})=M'\cdot i'\circ j_G\left(  P(\pi_{N+n, i}) \right) \in A'(F_{n, \mfp}). \]
\end{enumerate}
\end{definition}

\begin{proposition}		\Label{Lunch}
For every $n\geq d$, we have

\begin{eqnarray*} \Tr_{F_{n, \mathfrak p}/F_{n-d, \mathfrak p}} Q(\pi_{N+n, i}) &=& -p\cdot b_1\cdot \Tr_{F_{n-1, \mathfrak p}/F_{n-d, \mathfrak p}} Q(\pi_{N+n-1, i}) \\
&&- p^2 \cdot b_2 \cdot \Tr_{F_{n-2, \mathfrak p}/F_{n-d, \mathfrak p}} Q(\pi_{N+n-2, i})		\\
&& - \quad  \cdots \quad - p^d \cdot b_d\cdot Q(\pi_{N+n-d, i}).
\end{eqnarray*}
\end{proposition}
\begin{proof}
Note that $M'$ annihilates every torsion of $G(\OO_{F_{n-d, \mfp}})$. Thus, the claim follows immediately from Proposition~\ref{Despicable-Laundry-Machine}.
\end{proof}

\begin{definition}[Relaxed Selmer groups] \Label{Relaxed Selmer}
Where $L$ is a number field,
\[ \Selr(\bfA/L)\stackrel{def}= \ker \left( H^1(L, \bfA)\to \prod_{v\nmid p} \ds \frac{H^1(L_v, \bfA)}{H^1_f(L_v, \bfA)} \right)\]
where 

\[ H^1_f (L_v, \bfA)\stackrel{def}=H^1_{un}(L_v, \bfA)\stackrel{def}=H^1(L_v^{un}/L_v, \bfA^{G_{L_v^{un}}}).\]
\end{definition}
In fact, when ${G_{L_v^{un}}}$ acts trivially on $A$ (i.e., good reduction at $v$), $H^1_{un}(L_v, \bfA)$ is the standard definition for a local condition $H^1_f(L_v, A)$. (Local conditions for a finite number of primes not above $p$ do not affect our result.)

Set

\[ \Gamma\stackrel{def}=\Gal(F_{\infty}/F),\]
\[ \Lambda\stackrel{def}=\Zp[[\Gamma]]\cong \Zp[[X]] \]
where the last isomorphism is (non-canonically) given by choosing a topological generator $\gamma$ of $\Gamma$, and set $\gamma=X+1$.

\begin{assumption}		\Label{Fries}
Let $M^{\vee}$ denote the Pontryagin dual $\Hom(M,\rat/\Z)$. We assume
\[ \rank_{\Lambda} \Selr(\bfA/ F_{\infty})^{\vee}=[F_{\mfp}:\Qp]=e.  \]
\end{assumption}
If $\dim G$ is not $1$, then we probably need to multiply it to $e$ in Assumption~\ref{Fries}.
We can show Assumption~\ref{Fries} is true if $\Sel(\bfA/F)$ or $\Sel(\bfA/F_n)^{\chi}$ for some primitive character $\chi$ of $\Gal(F_n/F)$ is finite. Although there are some notable counterexamples to this assumption (for instance, when $F_{\infty}$ is the anti-cyclotomic extension), for all intents and purposes, it is a safe assumption.

Let 

\[ S_{tor}=\left( \Selr(\bfA/ F_{\infty})^{\vee} \right)_{\Lambda-torsion}.\]
If we assume Assumption~\ref{Fries}, then there is a short exact sequence

\begin{eqnarray}		\Label{Ice-Cream}	 0  \to \Selr(\bfA/F_{\infty})^{\vee}/S_{tor} \to \Lambda^e \to C \to 0
\end{eqnarray}
for a finite group $C$. 

\begin{notation}
\begin{enumerate}
\item For each $n\geq 0$,

\[ \Gamma_n=\Gamma/\Gamma^{p^n}, \quad \Lambda_n=\Zp[\Gamma_n]. \]
\item For a group $M$ on which $\Gamma$ acts,

\[ M_{/\Gamma^{p^n}}=M/\{ (1-a)\cdot m \; |\; a \in \Gamma^{p^n}, m \in M\}. \]
Equivalently, where $\gamma$ is a topological generator of $\Gamma$, 

\[ M_{/\Gamma^{p^n}}=M/(1-\gamma^{p^n})\cdot M. \]
\end{enumerate}
\end{notation}

\begin{lemma} Suppose there is an exact sequence of $\Lambda$-modules

\[ 0 \to A_1\to A_2 \to A_3 \to A_4 \to 0, \]
and $A_1$ and $A_4$ are finite. Then, for every $n$, the orders of the kernel and cokernel of

\[ \left( A_2 \right)_{/\Gamma^{p^n}} \to \left( A_3 \right)_{/\Gamma^{p^n}} \]
are bounded by $|A_1|\cdot|A_4|$.
\end{lemma}

\begin{proof}
The exact sequence induces two short exact sequences

\[ 0 \to A_1 \to A_2 \to A_2/A_1 \to 0,\]
\[ 0 \to A_2/A_1 \to A_3 \to A_4 \to 0, \]
which in turn induce

\[ (A_1)_{/\Gamma^{p^n}} \to (A_2)_{/\Gamma^{p^n}} \to (A_2/A_1)_{/\Gamma^{p^n}} \to 0,\]
\[ (A_4)^{\Gamma^{p^n}} \to (A_2/A_1)_{/\Gamma^{p^n}} \to (A_3)_{/\Gamma^{p^n}} \to (A_4)_{/\Gamma^{p^n}} \to 0.\]
Our claim follows immediately.
\end{proof}

It is not difficult to show $\Selr(\bfA/F_n) \to \Selr(\bfA/F_{\infty})^{\Gamma^{p^n}}$ has bounded kernel and cokernel for every $n$. For the sake of argument, we assume it is an isomorphism, which will not hurt the integrity of our argument. 

The map in (\ref{Ice-Cream}) induces the following:

\[
\alpha_n: 
(\Selr(\bfA/F_{\infty})^{\vee}/S_{tor})_{/\Gamma^{p^n}} \to \Lambda_n^e
\]
which induces

\[ \alpha_n': \Selr(\bfA/F_n)^{\vee} \to \Lambda_n^e \]
by the above assumption. We note that there is a map

\[ \beta_n:  A'(F_{n, \mathfrak p}) \to \Selr(\bfA/F_n)^{\vee} \]
given by the local Tate duality which states that $A'(F_{n, \mathfrak p})$ is the Pontryagin dual of $H^1(F_{n, \mathfrak p}, \bfA)/A(F_{n, \mathfrak p})\otimes \Qp/\Zp$.

\begin{definition}
\begin{enumerate}[(a)]
\item Let $R(\pi_{N+n, i}) \in \Lambda_n^e$ be the image of $Q(\pi_{N+n, i})$ under $\alpha_n' \circ \beta_n$.

\item Let $\Proj_n^m$ be the natural projection from $\Lambda_m$ to $\Lambda_n$ ($m\geq n$).
\end{enumerate}
\end{definition}

\bigskip

Let $H^{\vee}(X)= X^d+pb_1X^{d-1}+p^2b_2X^{d-2}+\cdots+p^db_d=0$. By Proposition~\ref{Lunch} we have

\begin{eqnarray}		\Label{Smolensk} \Proj_{n-d}^n R(\pi_{N+n, i}) +\sum_{k=1}^d p^k b_k \Proj_{n-d}^{n-k} R(\pi_{N+n-k, i})=0
\end{eqnarray}
for each $i$.

Here we recall Perrin-Riou's lemma: In the following $\Lambda_{\alpha}$ is the set of power series $f(T) \in \overline{\Q}_p[[T]]$ satisfying $|f(x)| < C |1/\alpha^n|$ for some fixed $C>0$ for every $n \geq 1$ and $x \in \C_p$ with $|x| < |1/\sqrt[p^n]p|$.

\begin{lemma}[\cite{Perrin-Riou-1}~Lemme~5.3.]	\Label{Perrin-Riou-Lemma}
Let $R(T)=\sum a_kT^k$ be a monic polynomial of $\Zp[T]$ whose roots are simple, non-zero, and have $p$-adic valuation strictly less than $1$. Suppose $f^{(n)}$'s are elements of $\Lambda$ satisfying the recurrence relation

\[ \sum_k a_k f^{(n+k)} \equiv 0 \pmod {(T+1)^{p^n}-1}. \]
Then, for every root $\alpha$ of $R(T)$, there is unique $f_{\alpha} \in \Lambda_{\alpha}$ so that for some fixed constant $c$, 

\[ f^{(n)}\equiv \sum_{\alpha} f_{\alpha} \alpha^{n+1} \pmod{c^{-1}((T+1)^{p^n}-1) \Lambda} \]
for every $n$.
\end{lemma}

\begin{proof} Simple linear algebra. See \cite{Perrin-Riou-1}.
\end{proof}

Since we assume $\bfF$ is a topological nilpotent on $M$, all the roots of $H^{\vee}(X)=0$ have $p$-adic valuation less than $1$.

Thus, by Lemma~\ref{Perrin-Riou-Lemma} and (\ref{Smolensk}), for each root $\alpha$ of $H^{\vee}(X)$, there is $f_{\alpha, i} \in \Lambda_{\alpha}^e$ associated to $\{ R(\pi_{N+n, i}) \}_n$.

\begin{definition}
Choose a generator $g_{tor}\in \Lambda$ of the characteristic ideal of $(\Selr(\bfA/F_{\infty})^{\vee})_{\Lambda-torsion}$. Then we let

\[ \bfL_{\alpha} \stackrel{def}= g_{tor}\times \det [f_{\alpha, 1}, \cdots,  f_{\alpha, e}].\]
\end{definition}

Suppose $\chi_n$ is a primitive character of $\Gal(F_n/F)$, and $\zeta_{p^n}=\chi_n (\gamma)$. Suppose $g_{tor}(\zeta_{p^n}-1)\not=0$ (true if $n$ is large enough). Then, we can see that

\vspace{3mm}
\begin{center}
``$\Sel_p(\bfA/F_n)^{\chi_n}$ is infinite $\leftrightarrow$ the $\chi_n$-part of the cokernel of $\alpha'\circ\beta_n$ is infinite

$\leftrightarrow$ $\{ R(\pi_{N+n,i})^{\chi_n} \}_{i=1,\cdots,e}$ generates a subgroup of $(\Lambda_n^e)^{\chi_n}$ of infinite index

$\longrightarrow$ $\left. \det [f_{\alpha, 1}, \cdots,  f_{\alpha, e}] \right|_{\gamma=\zeta_{p^n}} =0$.''
\end{center}

\vspace{2mm}

And, in such a case,
\begin{eqnarray}		\Label{Three Go}
\corank_{\Zp[\zeta_{p^n}]} \Sel_p(\bfA/F_n)^{\chi_n} \leq e.
\end{eqnarray}

Consider the following Perrin-Riou's lemma.

\begin{lemma}[\cite{Perrin-Riou-1}~Lemme~5.2.] Let $\lambda=v_p(\alpha)$.
Suppose $f \in \Lambda_{\alpha}$. Let $s_m$ be the number of positive integers $n$ ($n\leq m$) such that $f(\zeta_{p^n}-1)=0$ for every $p^n$-th primitive root of unity $\zeta_{p^n}$. If $s_m -\lambda m \to \infty$ as $m\to\infty$, then $f=0$.
\end{lemma}
In other words, $s_m =\lambda m +O(1)$ if $f \not=0$.

She assumed $0\leq \lambda <1$. But, in fact, since we assume $\bfF$ is a topological nilpotent, $\lambda<1$, so that condition is unnecessary.

We can modify Perrin-Riou's proof slightly, and obtain the following:

\begin{proposition}			\Label{ZeroGo}
\begin{enumerate}[(a)]
\item If $\bfL_{\alpha}\not=0$, then for some fixed $C$

\[ \corank_{\Zp} \Sel_p(\bfA/F_n) \leq e(p-1) \times \left\{ p^{n-1}+p^{n-2}+ \cdots+ p^m \right\}+C\]
where $n-m = \lambda n +O(1)$.

\item If any root $\alpha$ is a unit, then $\corank_{\Zp} \Sel_p(\bfA/F_n)$ is bounded by the number of roots of $\bfL_{\alpha}$ (counting multiplicity).

\end{enumerate}
\end{proposition}

\begin{proof}
Let $t_n$ be the number of the primitive $p^n$-th roots of unity which are roots of $\bfL_{\alpha}$.

By applying Perrin-Riou's proof for the above lemma, we get

\[ \sum_{m\leq n} t_m < e\lambda n +O(1). \]
Then we obtain our claim by (\ref{Three Go}).

If $\alpha$ is a unit, then $\bfL_{\alpha}$ is integral, so it has a finite number of roots. Thus, (b) is clear.
\end{proof}

This is a rough bound unless $H^{\vee}(X)$ has a unit root (i.e., unless the abelian variety has good ``in-between'' reduction). Probably it is possible to obtain a slightly better bound (ideally, something like ``$e(p-1)\times \{ p^{n-1}-p^{n-2}+\cdots \}$''), but not a substantially better one from $\bfL_{\alpha}$ alone, because any power series in $\Lambda_{\alpha}$ has an infinite number of roots. (For example, see R. Pollack's $\log_p^{\pm}$, \cite{Pollack}).

Thus, we need a new tool, and perhaps a new Selmer group. There is precisely such a tool in Sprung's $\sharp/\flat$-decomposition theory (\cite{Sprung}), and we will present our result in that direction in the next section.

Lastly, we want to discuss how Perrin-Riou obtained the result that $\rank E(\Q(\mu_{\infty}))$ is bounded. As stated above, it does not seem possible to obtain a finite bound from $\bfL_{\alpha}$ alone. However, she noted that her points $P_n \in E(\Q_{p, n})$ satisfy

\begin{eqnarray}			\Label{PR-Relations}
\Tr_{\Q_{p, n+1}/\Q_{p, n}}P_{n+1}-a_p P_n+P_{n-1}=0.
\end{eqnarray}
This is more sophisticatead than the relation $\Tr_{\Q_{p, n}/\Q_{p, n-2}}P_{n}-a_p \Tr_{\Q_{p, n-1}/\Q_{p, n-2}}P_{n-1}+p P_{n-2}=0$. She used these relations skilfully to obtain her result. Indeed, with the benefit of hindsight, we now know that recognizing such relations is the first step of the $\pm$-Iwasawa Theory, the $\sharp/\flat$-Iwasawa Theory, and so on.

In fact, in the next section, we will construct points satisfying relations analogous to (\ref{PR-Relations}), and use them to find a finite bound for $E(F_{\infty})$ where $F$ is ramified under some conditions. But, because the field is ramified, the relation will be given by matrices which vary depending on $n$.  
\end{subsection}
\end{section}

\begin{section}{Refined Local Points, Sprung's $\sharp/\flat$-Decomposition, and Finiteness of Ranks}		\Label{Case 2}
In this section, we consider only elliptic curves for simplicity. Take an elliptic curve $E$ over a number field $F$. Except that, our setting is the same as Section~\ref{Case 1}. But for readers' convenience, we will repeat our conditions and assumptions.

As in that section, we suppose

\begin{enumerate}
\item $k_{\infty}/\Qp$ is a totally ramified normal extension with $\Gal(k_{\infty}/\Qp)\cong \Z_p^{\times}$. By local class field theory, it is given by a Lubin-Tate group of height $1$ over $\Zp$. In other words, there is $\varphi(X)=X^p+\alpha_{p-1} X^{p-1}+\cdots+\alpha_1X \in \Zp[X]$ with $p|\alpha_i$, $v_p(\alpha_1)=1$ so that

\[ k_{\infty}=\cup_n \Qp(\pi_n) \]
where $\varphi(\pi_n)=\pi_{n-1}$ ($\pi_n\not=0$ for $n\geq 0$, $\pi_0=0$).

\item We let $F$ be a number field, and $F_{\infty}$ be a $\Zp$-extension of $F$. Since $E$ is an elliptic curve, its dual abelian variety is itself. Let $\bfT=T_pE$, and let $\bfA\stackrel{def}=\cup_n E[p^n]$.

\item We suppose there is only one prime $\mathfrak p$ of $F$ above $p$,  $\mathfrak p$ is totally ramified over $F_{\infty}/F$, $F_{\infty, \mathfrak p} = k_{\infty}$, and $F_{\mathfrak p}=\Qp(\pi_N)$ for some $N\geq 1$.

\item We set

\[ H(X)={\det}_{\Zp} (X\cdot 1_M-\bfF|M)=X^2-a_pX+p. \]
(Then, $a_p=1+N\mfp - \# \tilde E(\OO_{F_{\mfp}}/\mm_{\OO_{F_{\mfp}}})$). And, we set

\begin{eqnarray*} \bar H(X) \stackrel{def}= \ds \frac{H(X)}p &=& 1- \ds \frac{a_p}p X+\frac1p X^d 		\\
&=&1+b_1X+b_2X^{2}.
\end{eqnarray*}
\item We assume $E$ has \textbf{good supersingular reduction} at $\mathfrak p$. 

\end{enumerate}

\bigskip

\begin{subsection}{Fontaine's functor (revisited), and our assumptions.}
\hfill

Let $G$ be the formal group scheme given by the formal completion of $E/F_{\mfp}$, and let $M$ be the Dieudonne module of $G_{/\mathbb F_p}$, and $L$ be the set of logarithms of $G$. As in earlier sections, we choose a $\Zp[[\bfF]]$-generator $\bfm$ of $M$, and an $\OO_{F_{\mfp}}$-generator $\bfl$ of $L$.

Let $A'$ denote $\OO_{F_{\mfp}}$, and let $\mm$ denote its maximal ideal. Let $e$ be the ramification index of $F_{\mfp}$. (Since it is totally ramified, $e=[F_{\mfp}:\Qp]$.) Recall that $M_{A'}$ is the direct (i.e., injective) limit of

$$ \{ \mm^i \otimes M^{(j)} \}_{I_0}$$
where $I_0$ is the set of $(i,j) \in \Z\times \Z$ so that $j\geq 0$, and

$$ \left\{ 
\begin{array}{ll}
i\geq 0	&	\text{if } j=0,\\
i\geq p^{j-1}-je	&	\text{if }j\geq 1.
\end{array} \right.
$$
with maps $\varphi_{i,j}, f_{i,j}$, and $v_{i,j}$ between $\mm^i \otimes M^{(j)}$'s. Note that there is $s$ so that $p^s-(s+1)e \leq p^{j-1}-je$ for every $j\geq 1$.

\begin{proposition}		\Label{CA}
Let $E=p^s-(s+1)e$. There is a map 

\[ \iota: M_{A'} \to \mm^E \otimes M \]
which is well-defined, and its cokernel is finite. If $M_{A'}$ is torsion-free, $\iota$ is injective.
\end{proposition}
\begin{proof}
For each $\mm^i \otimes M^{(j)}$, we have a map $\mm^i \otimes M^{(j)}\to \mm^i \otimes M$ given by $f_{i,1}\circ f_{i,2}\circ \cdots \circ f_{i,j}$. Since $i \geq p^s-(s+1)e$, there is a map $\mm^i \otimes M\to \mm^{E} \otimes M$ given by $\varphi_{E+1,0} \circ \varphi_{E+2,0} \circ \cdots \circ \varphi_{i,0}$. The rest is clear.
\end{proof}

Then we can write

\[ \iota( \bfl)=\alpha_1 \bfm+\alpha_2 \bfF \bfm \]
for some $\alpha_1, \alpha_2 \in \mm^E$. We assume

\begin{assumption} 		\Label{Assumption K}
\[ p|\frac{\alpha_2}{\alpha_1}. \]
\end{assumption}
Doubtlessly, some formal groups associated to elliptic curves satisfy this condition, and many others do not. In fact, $\frac{\alpha_2}{\alpha_1}$ can have a negative $p$-adic valuation, although it is bounded below, and the bound depends on $e$.

Also we assume

\begin{assumption}	\Label{Assumption Torsions}
The group of torsions of $E(F_{\infty, \mfp})$ is finite.
\end{assumption}

This is a reasonable assumption. In fact, we can often show that $E[p]$ is irreducible as a $G_{F_{\mfp}}$-module.
\end{subsection}

\bigskip

\begin{subsection}{Finite bounds for ranks}		\Label{AlphaGo}
\begin{notation}		\Label{BetaGo}
\begin{enumerate}[(a)]
\item Where $e=[\Qp(\pi_N):\Qp]$, let $\{ \pi_{N,1},\cdots, \pi_{N, e}\}=\{ \pi_N^{\sigma} \}_{\sigma \in\Gal(\Qp(\pi_N)/\Qp)}$. 

\item Then, for every $n> N$ and for each $i=1,\cdots,e$, choose $\pi_{n,i}$ so that $\varphi(\pi_{n,i})=\pi_{n-1,i}$. (Then, $\pi_{n,i}$ is a uniformizer of $\Qp(\pi_n)$.)

\end{enumerate}
\end{notation}

Similar to Section~\ref{Some special} but slightly differently, we define the following.

\begin{definition}		\Label{MI}
\begin{enumerate}
\item
\[ J(X)=\bar H(X)-1=b_1X+b_2X^2=-\ds \frac {a_p}pX+\frac 1p X^2, \]

\item
\[ \epsilon=\ds \frac{\alpha_{p-1}}{p-a_p+1}, \]

\item
\[ l(X)=[1-J(\varphi)+J(\varphi)^2-\cdots ]\circ X \]
where $\varphi\circ X^n=\varphi(X)^n$,

\item Define $\tilde \bfx \in \Hom_{\Zp}(M, \mathcal P)$ given by

\[ \tilde \bfx(\bfm)=\epsilon+l(X)\]
\[ \tilde \bfx(\bfF \bfm)=\varphi \circ l(X).\]
\end{enumerate}
\end{definition}

\bigskip

We define the following functor.

\begin{definition}
We let $M'=\mm^E \otimes M$, and let $L'$ denote the maximal $A'$-submodule of $M'$ which contains $\iota(L)\subset M'$, and $L'/\iota(L)$ has a finite index.

\begin{enumerate}
\item	For an $A'$-algebra $g$, we define $G'(L',M)(g)$ as the set of $(u_{L'}, u_M)$ where

\[ u_{L'} \in \Hom_{A'}(L', \Q \otimes g),\]
\[ u_M \in \Hom_{\mathbf D} (M, CW(g/\mm g))\]
which naturally induces $u_M': M'(=\mm^E \otimes M) \to \mm^E \otimes CW(g/\mm g)$, so that $u_{L'}$ and $u_M'$ are identical under

\begin{eqnarray*}
\Hom_{A'\otimes \mathbf D}(M', \mm^E \otimes CW	(g/\mm g))&	\stackrel{\omega_g'}\longrightarrow	& \Hom_{A'}(L', (\Q \otimes g)/\mm^E \cdot P'(g))	\\
	&	& \qquad \uparrow		\\
	&	& \Hom_{A'} (L', \Q \otimes g).
\end{eqnarray*}

\item	Similarly, but slightly differently, define $G'(L', M)(A'[[X]])$ as the set of $(u_{L'}, u_M)$ where (in the following, $K'$ denotes $Frac(A')$)

\[ u_{L'} \in \Hom_{A'}(L', K'[[X]]),\]
\[ u_M \in \Hom_{\mathbf D} (M, CW(\mathbb F_p[[X]]))\]
which naturally induces $u_M': M' \to \mm^E \otimes CW(\mathbb F_p[[X]])$, so that $u_{L'}$ and $u_M'$ are identical under

\begin{eqnarray*}
\Hom_{A'\otimes \mathbf D}(M', \mm^E \otimes CW	(\mathbb F_p[[X]]))&	\longrightarrow	& \Hom_{A'}(L', K'[[X]] /\mm^E \cdot P'(A'[[X]]))	\\
	&	& \qquad \uparrow		\\
	&	& \Hom_{A'} (L', K'[[X]]).
\end{eqnarray*}

\item And, choose $M \in \Z(\geq 0)$ such that $p^M \cdot \mm^E \in A'$.
\end{enumerate}
\end{definition}

\bigskip

\begin{definition} 

\begin{enumerate}
\item	Recall $\tilde\bfx$ from Definition~\ref{MI}. Modulo $p\Zp[[X]]$, it induces $\bfx \in \Hom_{\mathbf D}(M, \overline{\mathcal P})$ satisfying

\[  \bfx(\bfm)=l(X) \pmod{p\Zp[[X]]}.\]
Then, define
\begin{enumerate}

\item We can choose an $A'$-generator $\bfl'$ of $L'$, and write it as

\[ \bfl'=\beta_1 \bfm+\beta_2 \bfF \bfm \]
where $\ds \frac{\beta_2}{\beta_1}=\frac{\alpha_2}{\alpha_1}$. Then, define $\bfy \in \Hom_{A'}(L', K'[[X]])$ by

\[ \bfy(\bfl)= \beta_1 \tilde \bfx(\bfm)+\beta_2 \tilde \bfx(\bfF \bfm)= \beta_1 \left( \epsilon + l(X) \right)+\beta_2 l(\varphi(X)) \]
and extend $A'$-linearly.

\item Then, we set
\[ P'=(\bfy, \bfx) \in G'(L', M)(A'[[X]]). \]
\end{enumerate}

\item	Then, for every $n \geq N$ and $i=1,2,\cdots, e$, we obtain points $P'(\pi_{n,i}) \in G'(L', M)(\Zp[\pi_n])$ by substituting $X=\pi_{n,i}$.

\end{enumerate}
\end{definition}

We make the following assumption analogous to \cite{Kobayashi}~Proposition~8.12~(ii).

\begin{assumption}		\Label{Assumption L}
$\{ P'(\pi_{n,1}), \cdots, P'(\pi_{n,e}) , P'(\pi_{n-1,1}), \cdots, P'(\pi_{n-1,e}) \}$ generates $G'(L', M)(\Zp[\pi_n])$ over $\Zp[\Gal(\Qp(\pi_n)/\Qp(\pi_N))]$ modulo torsions for every $n > N$.
\end{assumption}
We can apply the proof of \cite{Kobayashi} to this assumption certainly in some cases. We hope we can in most cases.

\bigskip

\begin{definition}We define a map $\xi:G'(L',M)\to G(L,M)$ as follows:
\begin{enumerate}[(a)]
\item First, recall that $G(L, M)(g)$ is the set of $(u_L, u_M)$ where $u_L:L\to \Qp\otimes g$ and $u_{M_{A'}}: M_{A'} \to CW_{k}(g/\mm g)_{A'}$ are identical through the diagram

\begin{eqnarray*}
\Hom_{A'\otimes \mathbf D}(M_{A'}, CW_k (g/\mm g)_{A'} )	&	\longrightarrow	& \Hom_{A'}(L, (\Qp\otimes g) /P'(g))	\\
	&	& \qquad \uparrow		\\
	&	& \Hom_{A'} (L, \Qp\otimes g).
\end{eqnarray*}
We also recall that $G'(L',M)(g)$ is the set of $(u_{L'}, u_M)$ where
$u_{L'} \in \Hom_{A'}(L', F)$, and $u_M \in \Hom_{\mathbf D} (M, CW(g/\mm g))$ 
which naturally induces $u_M': M'(=\mm^E \otimes M) \to \mm^E \otimes CW(g/\mm g)$ satisfy that $u_{L'}=u_M'$ through the diagram

\begin{eqnarray*}
\Hom_{A'\otimes \mathbf D}(M', \mm^E \otimes CW	(g/\mm g))	&		\longrightarrow	& \Hom_{A'}(L',  (\Qp\otimes g)/\mm^E P'(g))	\\
	&	& \qquad \uparrow		\\
	&	& \Hom_{A'} (L',  \Qp\otimes g).
\end{eqnarray*}

\item We recall that $\iota: M_{A'} \to M'(=\mm^E \otimes M)$ (which is identity on $M$) also induces $\iota:L\to L'(\supset \iota(L))$. Then, $p^M\cdot\iota^*$ induces 

\begin{eqnarray*}	
p^M\cdot\iota^*	&:& \Hom_{A'\otimes \mathbf D}(M', \mm^E \otimes CW	(g/\mm g)) \to \Hom_{A'\otimes \mathbf D}(M_{A'},  CW_{k} (g/\mm g)_{A'} ),		\\
p^M\cdot\iota^*	&:&		 \Hom_{A'} (L',  \Qp\otimes g) \to 	\Hom_{A'} (L,  \Qp\otimes g),	\\
p^M\cdot\iota^*	&:&		 \Hom_{A'} (L',  (\Qp\otimes g)/ \mm^E P'(g)) \to 	\Hom_{A'} (L,  (\Qp\otimes g)/P'(g)),
\end{eqnarray*}
because $p^M\cdot \mm^E \subset A'$. 
\end{enumerate}
Thus, $p^M\cdot \iota^*$ induces a map $\xi: G'(L', M)\to G(L, M)$.

\end{definition}

We also define:

\begin{definition}
Recall the isogeny $j_G: G(L, M) \to G$, and an embedding $i:G \to E$. We define

\[ P(\pi_{n,i})=i \circ j_G \circ \xi (P'(\pi_{n,i})) \]
for every $n \geq N$ and $i=1,2,\cdots, e$.
\end{definition}

\bigskip

\begin{proposition} 			\Label{Mark IV}
For now, let $\Tr_{n/m}$ denote $\Tr_{\Qp(\pi_n)/\Qp(\pi_m)}$. For $n>N$, we have

\[ \Tr_{n/n-1} \begin{bmatrix} P(\pi_{n, 1}) \\ \vdots \\ P(\pi_{n, e})	\end{bmatrix} = pA_{n-1} \begin{bmatrix} P(\pi_{n-1, 1}) \\ \vdots \\ P(\pi_{n-1, e})	\end{bmatrix}- A'_{n-1} \begin{bmatrix} P(\pi_{n-2, 1}) \\ \vdots \\ P(\pi_{n-2, e})	\end{bmatrix} \]
where $A_{n-1}$ is an $e\times e$ matrix with entries in $\Zp[\Gal(\Qp(\pi_{n-1})/\Qp(\pi_N))]$, and $A'_{n-1}$ is an $e\times e$ matrix also with entries in $\Zp[\Gal(\Qp(\pi_{n-1})/\Qp(\pi_N))]$ so that 

\[ A'_{n-1}\equiv I_e \pmod p. \]
\end{proposition}

\begin{proof} As in the proof of Proposition~\ref{Despicable-Laundry-Machine},
\begin{eqnarray*} \Tr_{n/n-1} l(\pi_{n,i}) = -\alpha_{p-1}-p\cdot \left[b_1 l(\pi_{n-1, i})+b_2 l(\pi_{n-2, i}) \right]
\end{eqnarray*}
thus

\begin{eqnarray} 		\Label{Falcon}
\Tr_{n/n-1} (\epsilon+l(\pi_{n, i}))-a_p(\epsilon+l(\pi_{n-1, i}))+(\epsilon+l(\pi_{n-2, i}))=0.
\end{eqnarray}

On the other hand, again as in the proof of Proposition~\ref{Despicable-Laundry-Machine},

\[ l(\varphi(\pi_{n, i}))=\pi_{n-1, i}-\left( \ds \frac {-a_p}p l(\pi_{n-2, i})+\frac 1p l(\pi_{n-3, i}) \right) \]
thus

\begin{multline} \Tr_{n/n-1} l(\varphi(\pi_{n, i})) = p \pi_{n-1, i}+a_pl(\pi_{n-2, i})-l(\pi_{n-3, i}) 	\\
= 	\Label{Eagle}	p\pi_{n-1, i} + a_p l(\varphi(\pi_{n-1, i}))-l(\varphi(\pi_{n-2, i})).
\end{multline}

Since $\ds \frac{\beta_2}{\beta_1} \pi_{n-1, i}$ is divisible by $p$, there is $d_{n-1,i} \in \mm_{\Zp[\pi_{n-1}]}$ so that

\[ (\epsilon+l(d_{n-1,i}))+\ds \frac{\beta_2}{\beta_1} l(\varphi(d_{n-1,i}))=\frac{\beta_2}{\beta_1} \pi_{n-1, i}. \]

In other words, $\bfy(\bfl)(d_{n-1,i})=\beta_2 \pi_{n-1, i}$. 

Let $D_{n-1, i}=P'(d_{n-1,i}) \in G'(L', M)(\Zp[\pi_{n-1}])$. 
By Assumption~\ref{Assumption L}, 

\begin{multline}	\Label{ABCD}
\begin{bmatrix} D_{n-1, 1}\\ \vdots \\ D_{n-1, e} \end{bmatrix}=
 \begin{bmatrix} a_{n-1, 11} &\cdots & a_{n-1, 1e} \\ \vdots & \ddots & \vdots \\ a_{n-1, e1} &\cdots & a_{n-1, ee} \end{bmatrix} \cdot \begin{bmatrix} P'(\pi_{n-1, 1}) \\ \vdots \\ P'(\pi_{n-1, e}) \end{bmatrix} \\
+ \begin{bmatrix} a_{n-1, 11}' &\cdots & a_{n-1, 1e}' \\ \vdots & \ddots & \vdots \\ a_{n-1, e1}' &\cdots & a_{n-1, ee}' \end{bmatrix} \cdot \begin{bmatrix} P'(\pi_{n-2, 1}) \\ \vdots \\ P'(\pi_{n-2, e}) \end{bmatrix}
\end{multline}
modulo torsions for some $a_{n-1, ij}, a_{n-1, ij}' \in \Zp[\Gal(\Qp(\pi_{n-1})/\Qp(\pi_N))]$.

For $Q=(y_Q, x_Q) \in G'(L', M)(g)$ with $y_Q \in \Hom_{A'}(L', \Qp\otimes g)$, we let $\bfl(Q)$ denote $y_Q(\bfl)$. For example, $\bfl(P'(\pi_{n,i}))=\beta_1(\epsilon+l(\pi_{n,i}))+\beta_2l(\varphi(\pi_{n,i}))$.

Then, by (\ref{Falcon}), (\ref{Eagle}), and (\ref{ABCD}),

\begin{eqnarray*} \Tr_{n/n-1} \begin{bmatrix} \bfl(P'(\pi_{n, 1})) \\ \vdots \\ \bfl(P'(\pi_{n, e}))	\end{bmatrix} &=& a_p \begin{bmatrix} \bfl(P'(\pi_{n-1, 1})) \\ \vdots \\ \bfl(P'(\pi_{n-1, e}))	\end{bmatrix}- \begin{bmatrix} \bfl(P'(\pi_{n-2, 1})) \\ \vdots \\ \bfl(P'(\pi_{n-2, e}))	\end{bmatrix}		\\
&& +pB_{n-1}\cdot \begin{bmatrix} \bfl(P'(\pi_{n-1, 1})) \\ \vdots \\ \bfl(P'(\pi_{n-1, e})) \end{bmatrix}+ p B_{n-1}'  \cdot \begin{bmatrix} \bfl(P'(\pi_{n-2, 1})) \\ \vdots \\ \bfl(P'(\pi_{n-2, e})) \end{bmatrix}
\end{eqnarray*}
where $B_{n-1}, B_{n-1}'$ are the matrices that appear in (\ref{ABCD}). Since $L'$ is one-dimensional, this implies an analogous identity for the $y$-part of $P'(\pi_{n,i})$'s, therefore an analogous identity for $P'(\pi_{n,i})$'s holds modulo torsions. 

By taking $ i\circ j_G \circ (p^M\cdot \iota^*)$, we obtain our claim because $p^M$ annihilates the torsions.
\end{proof}

This relation is finer than the one used in Section~\ref{Case 1}, and we will adopt Sprung's insight of $\sharp/\flat$-decomposition to produce the characteristics $\bfL^{\sharp}, \bfL^{\flat}$ which are integral power series, which make a big difference between this section and the previous one.

Recall that $e=[F_{\mfp}:\Qp]=[F:\Q]$. As in Section~\ref{Case 1}, we assume

\[ \rank_{\Lambda} \Selr(E[p^{\infty}]/F_{\infty})^{\vee}=e. \]

It is not difficult to prove this assumption when $\Sel_p(E[p^{\infty}]/F_n)^{\chi_n}$ is finite for some $n$ and a character $\chi_n$.

Also as in Section~\ref{Case 1}, we let

\[ S_{tor}=\left( \Selr(E[p^{\infty}]/F_{\infty})^{\vee} \right)_{\Lambda-torsion}. \]
As Section~\ref{Case 1}, there is

\begin{eqnarray}			 0 \to  \Selr(A/F_{\infty})^{\vee}/S_{tor} \to \Lambda^e \to C \to 0
\end{eqnarray}
for a finite group $C$. 

It is often true that $\Selr(E[p^{\infty}] /F_n) \to \Selr(E[p^{\infty}] /F_{\infty})^{\Gamma^{p^n}}$ is an isomorphism, and even when it is not, its kernel and cokernel are bounded, so are easy to deal with. In this section, for convenience, assume it is an isomorphism for each $n$. The above short exact sequence induces

\[ \alpha_n': \Selr(E[p^{\infty}]/F_n)^{\vee} \to  \left( \Selr(A/F_{\infty})^{\vee}/S_{tor} \right)_{/\Lambda^{p^n}} \to \Lambda_n^e.\]

\begin{definition}
\begin{enumerate}[(a)]
\item

Recall $P(\pi_{N+n, i})$ is a point of $E(\Qp(\pi_{N+n}))=E(F_{n, \mfp})$.

Recall $\Lambda_n=\Zp[\Gamma_n]$. Let $R(\pi_{N+n, i})$ be the image of $P(\pi_{N+n, i})$ under the map

\[ E(F_{n,\mfp}) \to \Selr(E[p^{\infty}]/F_n)^{\vee} \stackrel{\alpha_n'}\to \Lambda_n^e. \]

\item
Choose a lifting $\tilde R(\pi_{N+n, i}) \in \Lambda^e$ of $ R(\pi_{N+n, i})$ for each $n$. Our result will not depend on the choice of $\tilde R(\pi_{N+n, i})$. Let

\[ \mathbf R_{N+n}=\begin{bmatrix} \tilde R(\pi_{N+n, 1})^t \\ \vdots \\ \tilde R(\pi_{N+n, e})^t	\end{bmatrix} \in M_{e}(\Lambda) .\]

\item Let $\Phi_n \in \Lambda$ be the minimal polynomial of $\zeta_{p^n}-1$, i.e., $\Phi_n=\ds \frac{(1+X)^{p^n}-1}{(1+X)^{p^{n-1}}-1}$ if $n \geq 1$, and $\Phi_0=X$. And, let $\omega_n=(1+X)^{p^n}-1$. We consider $\Phi_n$ and $\omega_n$ as elements of $\Lambda$ under the identification $\Lambda=\Zp[[X]]$.
\end{enumerate}
\end{definition}

We note $\Phi_n=\sum _{\sigma \in \operatorname{Ker}(\Gamma_n \to \Gamma_{n-1})} \sigma \pmod{\omega_n}$.

\begin{proposition}		\Label{Montana}

\[ \begin{bmatrix} \mathbf R_{N+n+1}	\\  \mathbf{ R}_{N+n}		\end{bmatrix} = 
\begin{bmatrix} pA_{N+n} &	-A'_{N+n} \Phi_n  \\
I_e	&	0	\end{bmatrix} \cdot
\begin{bmatrix}	\mathbf R_{N+n}		\\
\mathbf{ R}_{N+n-1}
\end{bmatrix}
\pmod{\omega_n}.
\]
\end{proposition}
\begin{proof}
This follows immediately from Proposition~\ref{Mark IV}.
\end{proof}

\begin{definition}		\Label{Mateo}
We choose liftings $\tilde A_{N+n}, \tilde A'_{N+n} \in M_e(\Lambda)$ of $A_{N+n}, A'_{N+n}$ for every $n$. We set
\begin{eqnarray*} \begin{bmatrix}		\tilde \bfL^{\sharp}(E)	\\	\tilde \bfL^{\flat}(E)	\end{bmatrix} 
&\stackrel{def}=&	 
\varprojlim_n
\begin{bmatrix} p\tilde A_{N+1} &	-\tilde A'_{N+1} \Phi_1  \\
I_e	&	0	\end{bmatrix}^{-1} 
\cdot 
\begin{bmatrix} p \tilde A_{N+2} &	-\tilde A'_{N+2} \Phi_2  \\
I_e	&	0	\end{bmatrix}	^{-1}	
\cdot \\
&& \cdots\quad \cdot
		\begin{bmatrix} p \tilde A_{N+n} &	-\tilde A'_{N+n} \Phi_n  \\
I_e	&	0	\end{bmatrix}^{-1} \cdot
\begin{bmatrix} \mathbf R_{N+n+1}	\\  \mathbf{ R}_{N+n}		\end{bmatrix}.		\\
\bfLalg^{\sharp}(E)	&\stackrel{def}=&	\det(\tilde \bfL^{\sharp}(E)),	\\
\bfLalg^{\flat}(E)	&\stackrel{def}=&	\det(\tilde \bfL^{\flat}(E)).
\end{eqnarray*}
\end{definition}

\begin{proposition}	\Label{GagConcert}
(a) $\tilde \bfL^{\sharp}(E)$ and $\tilde \bfL^{\flat}(E)$ are well-defined (i.e., the projective limits exist), and (b) their entries are in $\Lambda$.
\end{proposition}

\begin{proof} 
First, we show the following:

Let $c_{n+1}=\mathbf R_{N+n+1}, d_{n+1}=\mathbf R_{N+n}$, and

\[ \begin{bmatrix} c_i \\ d_i \end{bmatrix} =\begin{bmatrix} p \tilde A_{N+i} &	-\tilde A'_{N+i} \Phi_i  \\
I_e	&	0	\end{bmatrix}^{-1} \cdots \begin{bmatrix} p \tilde A_{N+n} &	-\tilde A'_{N+n} \Phi_n  \\
I_e	&	0	\end{bmatrix}^{-1} \begin{bmatrix} \mathbf R_{N+n+1}  \\ \mathbf R_{N+n} \end{bmatrix} \]
for every $1 \leq i \leq n$. We will show that

\begin{enumerate}[(1)]
\item $c_i, d_i \in M_e(\Lambda)$,
\item $c_i \equiv \mathbf R_{N+i} \pmod{\omega_i}, d_i \equiv \mathbf R_{N+i-1} \pmod{\omega_{i-1}}$ for every $1\leq i\leq n+1$.
\end{enumerate}
We prove it inductively as follows:

\emph{Step 1.} By the definition of $c_{n+1}$ and $d_{n+1}$, the claim is true for $i=n+1$.

\emph{Step 2.} Suppose the claim is true for $c_{i+1}, d_{i+1}$. Then,

\begin{eqnarray*} \begin{bmatrix} c_i \\ d_i \end{bmatrix} &=& \begin{bmatrix} p \tilde A_{N+i} &	-\tilde A'_{N+i} \Phi_i  \\
I_e	&	0	\end{bmatrix}^{-1} \begin{bmatrix} c_{i+1} \\ d_{i+1} \end{bmatrix} 		\\
&=& \ds \frac1{\Phi_i} (\tilde A_{N+i}')^{-1} 
\begin{bmatrix} 0& \tilde A_{N+i}' \Phi_i \\ 
-I_e & p\tilde A_{N+i} \end{bmatrix} 
\begin{bmatrix} c_{i+1} \\ d_{i+1} 
\end{bmatrix} 		\\
&=& \begin{bmatrix} d_{i+1} \\
\ds \frac1{\Phi_i} (\tilde A_{N+i}')^{-1} ( -c_{i+1}+p \tilde A_{N+i} d_{i+1} )
\end{bmatrix}
\end{eqnarray*}
By the induction hypothesis and Proposition~\ref{Montana}, we have

\begin{eqnarray*}	-c_{i+1}+p \tilde A_{N+i} d_{i+1} &=& -\mathbf R_{N+i+1}+p\tilde A_{N+i} \mathbf R_{N+i} \pmod{\omega_i}		\\
&=& \tilde A_{N+i}' \Phi_i \mathbf R_{N+i-1} \pmod{\omega_i}.
\end{eqnarray*}
Thus,

\[ \ds \frac1{\Phi_i} (\tilde A_{N+i}')^{-1} \left[ -c_{i+1}+p \tilde A_{N+i} d_{i+1} \right] \equiv \mathbf R_{N+i-1} \pmod{\omega_{i-1}} .\]
Thus, $c_i=d_{i+1}\equiv \mathbf R_{N+i} \pmod{\omega_i}$, and $d_i\equiv \mathbf R_{N+i-1} \pmod{\omega_{i-1}}$, and $c_i, d_i \in \Lambda^e$. Inductively, $c_1,d_1 \in M_e( \Lambda)$.

Second, we show the following:
By the above, for any $m \geq n$,

\[ \begin{bmatrix} p\tilde A_{N+n+1}& -\tilde A_{N+n+1}' \Phi_{n+1} \\
I_e & 0
\end{bmatrix}^{-1} \cdots \begin{bmatrix} p\tilde A_{N+m}& -\tilde A_{N+m}' \Phi_m \\
I_e & 0
\end{bmatrix}^{-1}
\begin{bmatrix} \mathbf R_{N+m+1} \\ \mathbf R_{N+m}
\end{bmatrix}
= \begin{bmatrix} r_{N+n+1} \\ s_{N+n+1}
\end{bmatrix} \]
where $r_{N+n+1} \equiv \mathbf R_{N+n+1} \pmod {\omega_{n+1}}$, $s_{N+n+1} \equiv \mathbf R_{N+n} \pmod {\omega_{n}}$. Let

\[ \begin{bmatrix} e_{n+1} \\ e_n \end{bmatrix}= \begin{bmatrix} r_{N+n+1} \\ s_{N+n+1}
\end{bmatrix} - \begin{bmatrix} \mathbf R_{N+n+1} \\ \mathbf R_{N+n}
\end{bmatrix}, \]
then $e_{n+1} \equiv 0 \pmod {\omega_{n+1}}$, $e_n \equiv 0 \pmod{\omega_n}$.

Let 

\[ \begin{bmatrix} e_i \\ e_{i-1} \end{bmatrix} =\begin{bmatrix} p \tilde A_{N+i} &	-\tilde A'_{N+i} \Phi_i  \\
I_e	&	0	\end{bmatrix}^{-1} \cdots \begin{bmatrix} p \tilde A_{N+n} &	-\tilde A'_{N+n} \Phi_n  \\
I_e	&	0	\end{bmatrix}^{-1} \begin{bmatrix} e_{n+1}  \\ e_n \end{bmatrix} \]
for every $1 \leq i \leq n$.

For our immediate purpose, we devise the following way of counting the number of divisors of elements of $M_e(\Lambda)$. If $f=p$ or $f=\Phi_i$ for some $i$, and $f|a \in M_e(\Lambda)$, we say $f$ is a divisor of $a$. Any other irreducible polynomial that divides $a$ is ignored in our way of counting. To define the number of divisors of $a$, we count $p$ any number of times that $p$ divides $a$ (for example, if $p^k|a$, then $p$ is counted $k$ times towards the number of divisors), but we count each $\Phi_i$ that divides $a$ only once (for example, if $\Phi_i^k|a$, then $\Phi_i$ is counted only once towards the number of divisors). For example, if $p^3 (X^2+2)| a \in M_e(\Lambda)$, then $a$ has at least $3$ divisors ($p$ is counted $3$ times, and $X^2+2$ is not counted), and if $p\Phi_2^2 \Phi_3^2 |b$, then $b$ has at least $3$ divisors ($\Phi_2$ and $\Phi_3$ are each counted only once). 

If $a=\sum a_i$ for some $a_i \in M_e(\Lambda)$ with each $a_i$ having at least $k$ divisors, we say $a$ is a sum of elements, each of which has at least $k$ divisors.

Suppose $e_{i+1}$ is a sum of elements, each of which has at least $n_{i+1}$ divisors, and suppose $e_i$ is a sum of elements, each of which has at least $n_{i}$ divisors. And, suppose $\omega_{i+1}|e_{i+1}$ and $\omega_i | e_i$. Then,

\[ e_{i-1}=\ds \frac1{\Phi_i} \tilde A_{N+i}'^{-1} (-e_{i+1}+p\tilde A_{N+i} e_i) = \tilde A_{N+i}'^{-1}(- \ds \frac 1{\Phi_i} e_{i+1}+\frac p{\Phi_i} \tilde A_{N+i} e_i)\]
and $\frac 1{\Phi_i} e_{i+1}$ and $\frac p{\Phi_i} \tilde A_{N+i} e_i$ are respectively a sum of elements, each of which has at least $n_{i+1}-1$ divisors, and a sum of elements, each of which has at least $n_{i}$ divisors. Both are divisible by $\omega_{i-1}$. Thus, $e_{i-1}$ is a sum of elements, each of which has at least $\operatorname{min}(n_{i+1}-1, n_i)$ divisors, and is divisible by $\omega_{i-1}$.

Since $\omega_{n+1}|e_{n+1}$ and $\omega_n|e_n$, it is not difficult to see that $e_1$ and $e_0$ are sums of elements, each of which has at least $n/2$ divisors.

For $i\geq 1$, $\Phi_i\equiv 0 \pmod{(p, X^{p^{i-1}})}$, so when $0\leq \alpha_1 < \cdots < \alpha_{n'}$ for some $n'$,

\begin{eqnarray*}		p^j \Phi_{\alpha_1}\cdots \Phi_{\alpha_{n'}} &\equiv & p^j p^{n'-i} * \pmod{X^{p^{i-1}}}		\\
&=& p^{n'-i+j}	*		\pmod{X^{p^{i-1}}}
\end{eqnarray*}
($*$ indicates any element). Thus, it follows that 

\[ \begin{bmatrix} e_1	\\	e_0	\end{bmatrix}
 \equiv \begin{bmatrix} 0\\ 0 	\end{bmatrix}
 \pmod{p^{n/2-i}, X^{p^{i-1}}}.  \]
In other words,

\begin{eqnarray*}
\bfL_{n,m}&	\stackrel{def}=	& \begin{bmatrix} p\tilde A_{N+1}& -\tilde A_{N+1}' \Phi_{1} \\
I_e & 0
\end{bmatrix}^{-1} 
\cdots \begin{bmatrix} p\tilde A_{N+m}& -\tilde A_{N+m}' \Phi_m \\
I_e & 0
\end{bmatrix}^{-1}
\begin{bmatrix} \mathbf R_{N+m+1} \\ \mathbf R_{N+m}
\end{bmatrix}		\\
&&
-\begin{bmatrix} p\tilde A_{N+1}& -\tilde A_{N+1}' \Phi_{1} \\
I_e & 0
\end{bmatrix}^{-1} 
\cdots \begin{bmatrix} p\tilde A_{N+n}& -\tilde A_{N+n}' \Phi_n \\
I_e & 0
\end{bmatrix}^{-1}
\begin{bmatrix} \mathbf R_{N+n+1} \\ \mathbf R_{N+n}
\end{bmatrix}		\\
&=&
\begin{bmatrix} 0 \\ 0 \end{bmatrix}	
\pmod{p^{n/2-i}, X^{p^{i-1}}},
\end{eqnarray*}
so $\bfL_{n,m}$ converges to $0$ uniformly as $n,m \to \infty$.

Thus, we obtain our claim.
\end{proof}

\vspace{3mm}

In the proof of Proposition~\ref{GagConcert}, we see that there are $\mathbf R_{N+n}^{(m)}, \mathbf R_{N+n-1}^{(m)} \in M_e(\Lambda)$ so that $\mathbf R_{N+n}^{(m)} \equiv \mathbf R_{N+n} \pmod{\omega_n}, \mathbf R_{N+n-1}^{(m)}\equiv \mathbf R_{N+n-1} \pmod{\omega_{n-1}}$, and

\[ \begin{bmatrix} \mathbf R_{N+n}^{(m)} \\ \mathbf R_{N+n-1}^{(m)} \end{bmatrix} =\begin{bmatrix} p \tilde A_{N+n} &	-\tilde A'_{N+n} \Phi_n  \\
I_e	&	0	\end{bmatrix}^{-1} \cdots \begin{bmatrix} p \tilde A_{N+m} &	-\tilde A'_{N+m} \Phi_m  \\
I_e	&	0	\end{bmatrix}^{-1} \begin{bmatrix} \mathbf R_{N+m+1}  \\ \mathbf R_{N+m} \end{bmatrix}. \]
From Definition~\ref{Mateo},

\begin{multline*} \begin{bmatrix} p \tilde A_{N+n-1} &	-\tilde A'_{N+n-1} \Phi_{n-1}  \\
I_e	&	0	\end{bmatrix}
\cdot \cdots \cdot
\begin{bmatrix} p \tilde A_{N+2} &	- \tilde A'_{N+2} \Phi_2  \\
I_e	&	0	\end{bmatrix}
\cdot
\begin{bmatrix} p \tilde A_{N+1} &	-\tilde A'_{N+1} \Phi_1  \\
I_e	&	0	\end{bmatrix}
\cdot
\begin{bmatrix}	\tilde \bfL^{\sharp}(E)	\\	\tilde \bfL^{\flat}(E)	\end{bmatrix} \\
=  \varprojlim_m \begin{bmatrix} \mathbf R_{N+n}^{(m)}	\\  \mathbf{ R}_{N+n-1}^{(m)}		\end{bmatrix}			,
\end{multline*}
and for a primitive $p^n$-th root of unity $\zeta_{p^n}$, $\varprojlim \mathbf R_{N+n}^{(m)}|_{X=\zeta_{p^n}-1}= \mathbf R_{N+n}|_{X=\zeta_{p^n}-1}$.

Then, naturally we would hope for the following:
Let $\chi$ be a finite character of $\Gamma$ satisfying $\chi(\gamma)=\zeta_{p^{n}}$. We may also consider it as a character of $\Gal(F_n/F)$.
It is not hard to see that assuming $S_{tor}^{\chi}$ is finite, $\det(\mathbf R_{N+n}|_{X=\zeta_{p^n}-1})=0$ if and only if $\Sel_p(E[p^{\infty}]/F_n)^{\chi}$ is infinite. Since $\bfLalg^{\sharp}=\det(\tilde\bfL^{\sharp}(E))$ and $\bfLalg^{\flat}=\det(\tilde\bfL^{\flat}(E))$ are in $\Lambda$, and therefore have a finite number of roots, we would hope that it implies that $\Sel_p(E[p^{\infty}]/F_n)^{\chi}$ is infinite for a finite number of characters $\chi$. But, the author finds it a little difficult to show that because we may have $\det R_{N+n}|_{X=\zeta_{p^n}-1}=0$ even when $\bfLalg^{\sharp}(\zeta_{p^n}-1)\not=0$ and $\bfLalg^{\flat}(\zeta_{p^n}-1)\not=0$.

Instead, we make a more modest claim:

\begin{proposition}			\Label{DDT}
Suppose $\bfLalg^{\sharp}$ and $\bfLalg^{\flat}$ are not 0, and $a_p$ and $\frac{\beta_2}{\beta_1}$ are divisible by $p^T$ for some $T>0$. Suppose $\chi$ is a primitive character of $\Gamma_n$ for sufficiently large $n$. Also, suppose that

\begin{enumerate}[(a)]
\item if $n$ is odd, $p^S\nmid \bfLalg^{\sharp}$ for some $S$ with $S+\ds \frac{ep}{(p-1)^2}<T$, or

\item if $n$ is even, $p^{S'} \nmid \bfLalg^{\flat}$ for some $S'$ with $S'+\ds \frac{ep}{(p-1)^2}<T$.
\end{enumerate}
Then, $E(F_n)^{\chi}$ and $\Sha(E/F_n)[p^{\infty}]^{\chi}$ are finite.
\end{proposition}

\begin{proof}
First, we note $p^{T-1}|B_{i}$ and $p^{T-1}|B_{i}'$ for each $i$ where $B_{i}, B_{i}'$ are the matrices in the proof of Proposition~\ref{Mark IV}, thus $p^{T-1}|A_{i}$ and $A_{i}'\equiv I_e \pmod{p^T}$. Then, we can choose $\tilde A_i, \tilde A_i'$ so that $p^{T-1}|\tilde A_i, \tilde A_i' \equiv I_e \pmod{p^T}$.

Thus, if $n$ is odd, for $\zeta_{p^n}=\chi(\gamma)$,

\begin{multline*}
\left.		\begin{bmatrix} p \tilde A_{N+n-1}	&	-\tilde A'_{N+n-1}  \Phi_{n-1}   \\
I_e	&	0	\end{bmatrix}
\cdot \cdots \cdot
\begin{bmatrix} p \tilde A_{N+2} 	 &	-\tilde A'_{N+2} 	 \Phi_2   \\
I_e	&	0	\end{bmatrix}		
\cdot		
\begin{bmatrix} p \tilde A_{N+1}  &	-\tilde A'_{N+1}  \Phi_1   \\
I_e	&	0	\end{bmatrix} 	\right|_{X=\zeta_{p^n}-1}\\
\equiv 	
\begin{bmatrix} 0	& -\Phi_{n-1}(\zeta_{p^n}-1) I_e	\\	I_e&0	\end{bmatrix} \cdot \cdots \cdot \begin{bmatrix} 0	& -\Phi_1(\zeta_{p^n}-1) I_e		\\	I_e&0	\end{bmatrix}
=
\begin{bmatrix} a I_e&0	\\ 0&bI_e \end{bmatrix}			\pmod{p^T}
\end{multline*}
for some $a,b$ with $v_p(a), v_p(b) < p/(p-1)^2$, and if $n$ is even, it is congruent to $\begin{bmatrix} 0&a I_e \\ b I_e &0 \end{bmatrix}$.

Then, in case (a), $a \tilde \bfL^{\sharp}(\zeta_{p^n}-1)\equiv \mathbf R_{N+n}(\zeta_{p^n}-1) \pmod {p^T}$, and in case (b), $a\tilde \bfL^{\flat}(\zeta_{p^n}-1)\equiv  \mathbf R_{N+n}(\zeta_{p^n}-1) \pmod {p^T}$. If $n$ is sufficiently large, $v_p(a^e \bfLalg^{\sharp}(\zeta_{p^n}-1))<T$ and $v_p(a^e \bfLalg^{\flat}(\zeta_{p^n}-1))<T$ respectively by our assumption, thus $\det(\mathbf R_{N+n}(\zeta_{p^n}-1)) \not\equiv 0 \pmod{p^T}$, and also $S_{tor}^{\chi}$ is finite for a sufficiently large $n$. Thus our claim follows.
\end{proof}

Then we immediately have:

\begin{theorem}			\Label{DDR}
Suppose 

\begin{enumerate}
\item $a_p$ and $\frac{\beta_2}{\beta_1}$ are divisible by $p^T$ for some $T>0$,

\item $p^S\nmid \bfLalg^{\sharp}, p^{S} \nmid \bfLalg^{\flat}$ for some $S$ with $S+\ds \frac{ep}{(p-1)^2}<T$.
\end{enumerate}
Then, $E(F_{\infty})/E(F_{\infty})_{tor}$ is a group of finite rank, and $\Sha(E/F_n)[p^{\infty}]^{\chi}$ is finite for all sufficiently large $n$, and every primitive character $\chi$ of $\Gal(F_n/F)$.
\end{theorem}
We note that it is often relatively easy to show that $E(F_{\infty})$ has a finite number of $p$-power torsions.

\bigskip
\end{subsection}

\begin{subsection}{Appendix: Sprung's $\sharp/\flat$-Selmer groups}
\Label{Appendix}

\hfill

Even though we do not use them in this paper, using the points constructed in Section~\ref{AlphaGo}, we can construct $\Sel_p^{\sharp}(E/F_{\infty})$ and $\Sel_p^{\flat}(E/F_{\infty})$ as Sprung did (\cite{Sprung}).

\begin{definition}[Perrin-Riou map]
\begin{enumerate}
\item  Let $(\cdot, \cdot)_{N+n}$ denote the following pairing given by the local class field theory:

\[ (\cdot, \cdot)_{N+n}: H^1(\Qp(\pi_{N+n}), T_p) \times H^1(\Qp(\pi_{N+n}), T_p) \to \Zp.\]

Recall $\Gamma_n=\Gal(F_n/F) \cong \Gal(\Qp(\pi_{N+n})/ \Qp(\pi_N))$, $\Gamma=\varprojlim \Gamma_n$, and $\Lambda=\Zp[[\Gamma]]\cong \Zp[[X]]$ (non-canonically).
For $z \in H^1(\Qp(\pi_{N+n}), T_p)$ and $x = [x_1,\cdots, x_e]^t \in E(\Qp(\pi_{N+n}))^e$,

\[ \bfP_{N+n, x}(z)\stackrel{def}= \begin{bmatrix} \sum_{\sigma \in \Gamma_n} (z, x_1^{\sigma})_{N+n} \cdot \sigma	\\
\sum_{\sigma \in \Gamma_n} (z, x_2^{\sigma})_{N+n} \cdot \sigma \\
\vdots \\
\sum_{\sigma \in \Gamma_n} (z, x_e^{\sigma})_{N+n} \cdot \sigma
\end{bmatrix} 
\in \Zp[\Gamma_n]^e.\]

\item Also, let $\tilde \bfP_{N+n, x}(z)$ denote its lifting to $\Zp[\Gamma_{n+1}]^e$.
\end{enumerate}
\end{definition}

\begin{notation}
\begin{enumerate}
\item Let $x_{N+n}$ denote

\[ x_{N+n}=[ P(\pi_{N+n, 1}),		\cdots,	P(\pi_{N+n, e}) ]^t.
\]

\item Let $\Proj_{n/m}$ denote the natural projection from $\Zp[\Gamma_n]$ to $\Zp[\Gamma_m]$.
\end{enumerate}
\end{notation}

\bigskip

By Proposition~\ref{Mark IV}, for any $z=(z_n) \in \varprojlim_{n \geq N} H^1(\Qp(\pi_n), T_p)$,

\[ \Proj_{n+1/n} \begin{bmatrix} \bfP_{N+n+1, x_{N+n+1}}	(z_{N+n+1})		\\ \tilde \bfP_{N+n, x_{N+n}} (z_{N+n})		\end{bmatrix} = 
\begin{bmatrix} pA_{N+n} &	-A'_{N+n} \Phi_n  \\
I_e	&	0	\end{bmatrix} \cdot
\begin{bmatrix}	\bfP_{N+n, x_{N+n}}	(z_{N+n})		\\
\tilde\bfP_{N+n-1, x_{N+n-1}}	(z_{N+n-1})
\end{bmatrix}
\]

Following Sprung (\cite{Sprung}), we can define the following:

\begin{definition}

From the previous section we recall the liftings $\tilde A_{N+n}, \tilde A'_{N+n} \in M_e(\Lambda)$ of $A_{N+n}, A'_{N+n} \in M_e(\Lambda_n)$ for every $n$.

For $z=(z_n) \in \varprojlim_{n \geq N} H^1(\Qp(\pi_n), T_p)$,
\begin{eqnarray*} \begin{bmatrix}		\Col^{\sharp}(z)	\\	\Col^{\flat}(z)	\end{bmatrix} 
&\stackrel{def}=&	\varprojlim_n
 \begin{bmatrix} p \tilde A_{N+1} &	-\tilde A'_{N+1} \Phi_1  \\
I_e	&	0	\end{bmatrix}^{-1}
 \cdot 
\begin{bmatrix} p \tilde A_{N+2} &	-\tilde A'_{N+2} \Phi_2  \\
I_e	&	0	\end{bmatrix}^{-1}
 \cdot \\
&& \cdots\quad \cdot
\begin{bmatrix} p \tilde A_{N+n} &	-\tilde A'_{N+n} \Phi_n  \\
I_e	&	0	\end{bmatrix}^{-1} \cdot
\begin{bmatrix} \bfP_{N+n+1, x_{N+n+1}}(z_{N+n+1})		\\ \tilde \bfP_{N+n, x_{N+n}}(z_{N+n})		\end{bmatrix}.
\end{eqnarray*}
\end{definition}

Similar to Proposition~\ref{GagConcert}, we can show $\Col^{\sharp}(z), \Col^{\flat}(z) \in \Lambda^e$. We omit its proof.

\begin{definition}		\Label{Sharp Distinction}

We recall the definition of the relaxed Selmer group $\Selr$ from Definition~\ref{Relaxed Selmer}.
We define

\[ \Sel_p^{\sharp}(E[p^{\infty}] /F_{\infty}) \stackrel{def}= \ker\left( \Selr(E[p^{\infty}] /F_{\infty}) \to \ds \frac {H^1(F_{\infty, \mfp}, E[p^{\infty}])}{\left( \ker \Col^{\sharp}\right)^{\perp}} \right) \]
where $\left( \ker \Col^{\sharp}\right)^{\perp}$ denotes the orthogonal complement of $\ker \Col^{\sharp}$ with respect to the local pairing $\varprojlim_n H^1(\Qp(\pi_n), T_p) \times H^1(\Qp(\pi_{\infty}), E[p^{\infty}])\to \Qp/\Zp$.

Similarly, we define $\Sel_p^{\flat}(E/F_{\infty})$.
\end{definition}

It seems likely that $\Sel_p^{\sharp}(E[p^{\infty}]/F_{\infty})$ and $\Sel_p^{\flat}(E[p^{\infty}]/F_{\infty})$ are $\Lambda$-cotorsion under some suitable assumptions. In fact, we can imagine

\begin{eqnarray}		\Label{Speculation One} 
char(\Sel_p^{\sharp}(E[p^{\infty}] /F_{\infty})^{\vee})=char(S_{tor})\cdot (\bfLalg^{\sharp}),
\end{eqnarray}

\begin{eqnarray}	\Label{Speculation Two}
(\text{resp.} \quad char(\Sel_p^{\flat}(E[p^{\infty}] /F_{\infty})^{\vee})=char(S_{tor}) \cdot (\bfLalg^{\flat}).)
\end{eqnarray}

But, the ways that $\Sel_p^{\sharp/\flat}$ and $\bfLalg^{\sharp/\flat}$ are defined seem to be dual to each other. Thus, we suspect that to prove such an equality, we may need some kind of self-duality (similar to the Tate local duality) of the local conditions such as the one proven in \cite{Kim-1}. The author cannot say with certainty that such self-duality exists for the local conditions in Definition~\ref{Sharp Distinction}, but, an analogous result has been proven for a different but related Selmer group (\cite{Lei-Ponsinet}), and the author is hopeful that equalities such as (\ref{Speculation One}) and (\ref{Speculation Two}) will be proven soon.

\end{subsection}

\end{section}

\end{document}